\documentclass[11pt, reqno]{article}
\usepackage{amsmath,amssymb,amsfonts,amsthm}
\usepackage{color}

\usepackage[colorlinks=true,linkcolor=blue,citecolor=blue,urlcolor=blue,pdfborder={0 0 0}]{hyperref}
\usepackage{cleveref}

\usepackage{caption}
\usepackage{thmtools}

\usepackage{mathrsfs}

\crefname{theorem}{Theorem}{Theorems}
\crefname{thm}{Theorem}{Theorems}
\crefname{lemma}{Lemma}{Lemmas}
\crefname{lem}{Lemma}{Lemmas}
\crefname{remark}{Remark}{Remarks}
\crefname{prop}{Proposition}{Propositions}
\crefname{defn}{Definition}{Definitions}
\crefname{corollary}{Corollary}{Corollaries}
\crefname{conjecture}{Conjecture}{Conjectures}
\crefname{question}{Question}{Questions}
\crefname{chapter}{Chapter}{Chapters}
\crefname{section}{Section}{Sections}
\crefname{figure}{Figure}{Figures}

\theoremstyle{plain}
\newtheorem{thm}{Theorem}[section]

\newtheorem{lemma}[thm]{Lemma}

\newtheorem{corollary}[thm]{Corollary}

\newtheorem{prop}[thm]{Proposition}
\newtheorem{conjecture}[thm]{Conjecture}

\newtheorem{question}[thm]{Question}
\theoremstyle{definition}

\theoremstyle{remark}
\newtheorem{remark}[thm]{Remark}

\numberwithin{equation}{section}

\renewcommand{\P}{\mathbb P}
\newcommand{\E}{\mathbb E}

\newcommand{\R}{\mathbb R}
\newcommand{\Z}{\mathbb Z}

\newcommand{\cO}{\mathcal O}

\newcommand{\sA}{\mathscr A}
\newcommand{\sB}{\mathscr B}

\usepackage[margin=1in]{geometry}
\usepackage{extarrows}
\usepackage{commath}
\usepackage{mathtools}
\newcommand{\eps}{\varepsilon}
\usepackage{bbm}
\usepackage{setspace}
\usepackage{enumitem}
\usepackage{tikz}

\usepackage{cite}

\newcommand{\Aut}{\operatorname{Aut}}

\newcommand{\bP}{\mathbf P}
\newcommand{\bE}{\mathbf E}

\newcommand{\lrDini}{\left(\frac{d}{dp}\right)_{\hspace{-0.2em}+}\!}

\newcommand{\urDini}{\left(\frac{d}{dp}\right)^{\hspace{-0.2em}+}\!}

\newcommand{\stab}{\operatorname{Stab}}

\newcommand{\opleq}{\preccurlyeq}

\title{\textsc{Percolation on Hyperbolic Graphs}}

\author{Tom Hutchcroft\footnote{Statistical Laboratory, DPMMS, University of Cambridge. Email: \href{mailto:t.hutchcroft@maths.cam.ac.uk}{t.hutchcroft@maths.cam.ac.uk}}}

\begin{document}

\maketitle

\setstretch{1.1}

\begin{abstract}
We prove that Bernoulli bond percolation on any nonamenable,  Gromov hyperbolic, quasi-transitive graph has a phase in which there are infinitely many infinite clusters, verifying a well-known conjecture of Benjamini and Schramm (1996) under the additional assumption of hyperbolicity. In other words, we show that $p_c<p_u$ for any such graph. Our proof also yields that the triangle condition $\nabla_{p_c}<\infty$ holds at criticality on any such graph, which is known to imply that several critical exponents exist and take their mean-field values. 
This gives the first 
family of examples of one-ended groups all of whose Cayley graphs are proven to have mean-field critical exponents for percolation.

\end{abstract}
\setstretch{1.13}

\section{Introduction}

 In \textbf{Bernoulli bond percolation}, the edges of a connected, locally finite graph $G$ are chosen to be either \textbf{open} or \textbf{closed} independently at random, with probability $p$ of being open. The subgraph of $G$ obtained by deleting all closed edges and retaining all open edges is denoted by $G[p]$. The connected components of $G[p]$ are referred to as \textbf{clusters}.
The \textbf{critical parameter} is defined to be
\[
p_c=p_c(G)=\inf\Bigl\{ p\in [0,1] : G[p] \text{ contains an infinite cluster almost surely}\Bigr\},
\]
and the \textbf{uniqueness threshold} is defined to be
\[
p_u=p_u(G)=\inf\Bigl\{ p\in [0,1] : G[p] \text{ contains a unique infinite cluster almost surely}\Bigr\}.
\]
Questions of central interest concern the equality or inequality of these two values of $p$ and the behaviour of percolation at and near $p_c$ and $p_u$. These questions were traditionally studied primarily on Euclidean lattices such as the hypercubic lattice $\Z^d$,
for which Aizenman, Kesten, and Newman \cite{MR901151} proved that $\Z^d[p]$ has at most one infinite cluster almost surely for every $p$, and hence that $p_c(\Z^d) =p_u(\Z^d)$ for every $d\geq 1$. A very short proof of the same result was later obtained by Burton and Keane \cite{burton1989density}. For further background on percolation, we refer the reader to \cite{grimmett2010percolation,1707.00520,heydenreich2015progress,LP:book}.

 The following conjecture was made in the highly influential paper of Benjamini and Schramm \cite{bperc96}, who proposed a systematic study of percolation on \emph{quasi-transitive} graphs, that is, graphs for which the action of the automorphism group on the vertex set  has at most finitely many orbits.

\begin{conjecture}[Benjamini and Schramm 1996]
\label{conj:pcpu}
Let $G$ be a connected, locally finite, quasi-transitive graph. Then $p_c(G)<p_u(G)$ if and only if $G$ is nonamenable. 
\end{conjecture}

Here, a graph $G$ is said to be \textbf{nonamenable} if its \textbf{Cheeger constant}
\vspace{-0.5em}
\[
h(G) = \inf\left\{\frac{|\partial_E K|}{\sum_{v\in K} \deg(v)} : K \subseteq V \text{ finite} \right\}
\]
is positive, 
where $\partial_E K$ denotes the set of edges with exactly one end-point in $K$. We say that $G$ is \textbf{amenable} if it is not nonamenable, i.e., if its Cheeger constant is zero.

 The proof of Burton and Keane was generalised by Gandolfi, Keane, and Newman \cite{gandolfi1992uniqueness} to show that $p_c(G)=p_u(G)$ for every amenable quasi-transitive graph $G$, so that only the `if' direction of \cref{conj:pcpu} remains to be settled. 
H\"aggstr\"omm, Peres, and Schonmann \cite{MR1676835,MR1676831,HPS99} proved that if $G$ is quasi-transitive then $G[p]$ has a unique infinite cluster almost surely for every $p>p_u$. Thus, since critical percolation on any quasi-transitive graph of exponential growth has no infinite clusters almost surely \cite{BLPS99b,timar2006percolation,Hutchcroft2016944}, a quasi-transitive graph $G$ has $p_c(G)<p_u(G)$ if and only if there exists some $p$ such that $G[p]$ has infinitely many infinite clusters almost surely. See \cite{Haggstrom06} for a survey of progress on \cref{conj:pcpu} and related problems.

A further folk conjecture is that critical percolation on any nonamenable quasi-transitive graph satisfies the \emph{triangle condition}, a sufficient condition for mean-field critical behaviour that was introduced by Aizenman and Newman \cite{MR762034} and proven to hold on high-dimensional Euclidean lattices by Hara and Slade \cite{MR1043524} (see also \cite{fitzner2015nearest}). We let $\tau_p(u,v)$ be the \textbf{two-point function}, i.e., the probability that $u$ and $v$ are connected in $G[p]$, and define the \textbf{triangle diagram} to be
\[
\nabla_p = \sup_{v\in V} \sum_{x,y\in V} \tau_p(v,x)\tau_p(x,y)\tau_p(y,v).
\]
The condition $\nabla_{p_c}<\infty$ is known as the \textbf{triangle condition}, and is known to imply that several critical exponents describing the behaviour of percolation at and near $p_c$ take their mean-field values, see \cref{corollary:criticalexponents} below.

\begin{conjecture}
\label{conj:triangle}
Let $G$ be a connected, locally finite, nonamenable, quasi-transitive graph. Then $\nabla_{p_c}<\infty$.
\end{conjecture}

The principal result of this paper is to establish \cref{conj:pcpu,conj:triangle} under the assumption that the graph in question is \emph{Gromov hyperbolic}. This is a geometric condition, which can be interpreted, very roughly, as meaning that the graph is negatively curved at large scales. We prove our theorems under the additional assumption of \emph{unimodularity}, the nonunimodular case having already been treated in~\cite{Hutchcroftnonunimodularperc}. These results apply in particular to lattices in $\mathbb{H}^d$ for $d\geq2$, for which \cref{thm:pcpu} was previously known only for $d=2$ and \cref{thm:triangle} is new for all $d\geq 2$. Gromov hyperbolicity is invariant under rough isometry \cite[Theorem 22.2]{Woess}, and \cref{thm:triangle} gives the first family of examples of one-ended finitely generated groups \emph{all} of whose Cayley graphs have mean-field critical exponents for percolation. 

\begin{thm}
\label{thm:pcpu}
Let $G$ be a connected, locally finite,  nonamenable, Gromov hyperbolic, quasi-transitive graph. Then $p_c(G)<p_u(G)$.
\end{thm}

\begin{thm}
\label{thm:triangle}
Let $G$ be a connected, locally finite,  nonamenable, Gromov hyperbolic, quasi-transitive graph. Then $\nabla_{p_c}<\infty$.
\end{thm}

Here, a graph is said to be \textbf{Gromov hyperbolic} \cite{Gromov81,Gromov87} if it satisfies the \textbf{Rips thin triangles property},
meaning that there exists a constant $C$ such that for any three vertices $u,v,w$ of $G$ and any three geodesics $[u,v]$, $[v,w]$ and $[w,u]$ between them, every point in the geodesic $[u,v]$ is contained in the union of the $C$-neighbourhoods of the geodesics $[v,w]$ and $[w,u]$. For example, every tree is hyperbolic, as it satisfies the Rips thin triangles property with constant $C=0$.  A finitely generated group is said to be (Gromov) hyperbolic if one (and hence all) of its Cayley graphs are Gromov hyperbolic. Every infinite, quasi-transitive Gromov hyperbolic graph is either nonamenable or rough-isometric to $\Z$. Bonk and Schramm \cite{MR1771428} proved that a bounded degree graph is Gromov hyperbolic if and only if it admits a rough-isometric embedding into real hyperbolic space $\mathbb{H}^d$ for some $d\geq 1$, a result that will be used extensively throughout this paper.

Many finitely generated groups and quasi-transitive graphs are Gromov hyperbolic. Examples include lattices in $\mathbb{H}^d$; random groups below the collapse transition \cite{MR2205306,MR3551187,MR3692901,MR3239613,MR1995802}; small cancellation $1/6$ groups \cite[\S 0.2A]{Gromov87}; fundamental groups of compact Riemannian manifolds of negative sectional curvature \cite[Chapter 3]{ghys1990groupes}; 
quasi-transitive graphs rough-isometric to simply connected Riemannian manifolds of sectional curvature bounded from above by a negative constant \cite[Chapter 3]{ghys1990groupes}; quasi-transitive CAT$(-k)$ graphs for $k>0$ \cite[Chapter 3]{ghys1990groupes}; and quasi-transitive, nonamenable, simply connected planar maps \cite{MR3658330}. Many surveys and monographs on hyperbolic groups have been written, and we refer the reader to e.g.\ \cite{Gromov87,ghys1990groupes,MR2243589} for further background. The specific background on hyperbolic geometry needed for the proofs of this paper is reviewed in \cref{sec:geometry}.

We remark that finitely generated hyperbolic groups are always finitely presented \cite[Chapter 4]{ghys1990groupes}, and it follows from the work of Babson and Benjamini \cite{MR1622785} that their Cayley graphs have $p_u<1$ if and only if they are one-ended. Other properties of percolation on lattices in $\mathbb{H}^d$ have been investigated in \cite{MR1873136,MR2986821,1804.05948}.


\medskip

Previous progress on \cref{conj:pcpu,conj:triangle} can  be briefly summarised as follows. First, several works \cite{MR1833805,MR1756965,MR3005730,MR3572426} have established \emph{perturbative criteria} under which $p_c<p_u$ and $\nabla_{p_c}<\infty$. In these papers, the assumption of nonamenability is replaced by a stronger quantitative assumption, for example that the Cheeger constant is large, under which it can be shown that $p_c<p_u$ and $\nabla_{p_c}<\infty$ by combinatorial methods. In particular, Pak and Smirnova-Nagnibeda~\cite{MR1756965} proved that every nonamenable finitely generated group has a Cayley graph for which $p_c<p_u$ (see also \cite{MR3352259}). Papers that apply perturbative techniques to study specific examples, including some specific examples of hyperbolic lattices, include \cite{tykesson2007number,1303.5624,yamamoto2017upper}. 

Let us now discuss \emph{non-perturbative} results. Benjamini and Schramm \cite{BS00} showed that $p_c<p_u$ for every \emph{planar} nonamenable quasi-transitive graph (see also \cite{unimodular2}), generalizing earlier work of Lalley \cite{MR1614583}. The later work of Gaboriau \cite{MR2221157} and Lyons \cite{MR1757952,MR3009109} used the ergodic-theoretic notion of \emph{cost} to prove that $p_c<p_u$ on any quasi-transitive graph admitting non-constant harmonic Dirichlet functions, a class that includes all those examples treated by \cite{MR1614583,BS00}. This property is invariant under rough isometry, and was until now the only condition (other than the conjecturally equivalent property of having cost $>1$) under which a finitely generated group was known to have $p_c<p_u$ for \emph{all} of its Cayley graphs. In particular, this result applies to every quasi-transitive graph rough isometric to $\mathbb{H}^2$, but does not apply to higher dimensional hyperbolic lattices \cite[Theorem 9.18]{LP:book}. Kozma \cite{kozma2011percolation} proved that $\nabla_{p_c}<\infty$ for the product of two three-regular trees, the first time that the triangle condition had been established by non-perturbative methods in a non-trivial example. (This example was recently revisited in \cite{1712.04911}.) Schonmann \cite{MR1888869} proved, without verifying the triangle condition, that several mean-field exponents hold 
on every transitive nonamenable planar graph and every infinitely ended, unimodular  transitive  graph. Similar results for certain specific lattices in $\mathbb{H}^3$ were proved by Madras and Wu \cite{madras2010}.  Very recently, we established that $p_c<p_u$ and $\nabla_{p_c}<\infty$ for every graph whose automorphism group has a \emph{quasi-transitive nonunimodular subgroup} \cite{Hutchcroftnonunimodularperc}. This  was until now the only setting in which both $p_c<p_u$ and $\nabla_{p_c}<\infty$ were established under non-perturbative hypotheses.
 In a different direction, in \cite{1710.03003} counterexamples were  constructed to show that the natural generalization of \cref{conj:pcpu} to \emph{unimodular random rooted graphs} is \emph{false}.

 The class of examples treated in this paper is mostly disjoint from the class treated in \cite{Hutchcroftnonunimodularperc}, and the methods we employ here are also very different to those of that paper. 
 Indeed, it follows from \cite[Theorem H]{de2015characterizing} that every hyperbolic group has a Cayley graph whose  automorphism group is discrete, and consequently does not have any nonunimodular subgroups.

\medskip

\cref{thm:triangle} has the following consequences regarding percolation at and near $p_c$. These consequences were established for the hypercubic lattice in the papers \cite{aizenman1987sharpness,MR762034,MR1127713,MR923855,MR2551766,MR2748397}. A detailed overview of how to adapt these proofs to the general quasi-transitive setting is given in \cite[\S 7]{Hutchcroftnonunimodularperc}. We write $\asymp$ for an equality that holds up to multiplication by a function that is bounded between two positive constants in the vicinity of the relevant limit point. 

\begin{corollary}[Mean-field critical exponents]
\label{corollary:criticalexponents}
Let $G$ be a connected, locally finite,  nonamenable, Gromov hyperbolic, quasi-transitive graph. Then the following hold for each $v\in V$.
\begingroup
\addtolength{\jot}{0.5em}
\begin{align}
\chi_p(v) &\asymp (p_c-p)^{-1}  &p &\nearrow p_c
\label{exponent:susceptibility}
\\
\chi_{p}^{(k+1)}(v)/\chi_{p}^{(k)}(v) &\asymp (p_c-p)^{-2} &k\geq 1,\, p &\nearrow p_c
\label{exponent:gap}
\\
\bP_{p}\left(|K_v|=\infty\right) &\asymp p-p_c &p &\searrow p_c
\label{exponent:theta}
\\
\bP_{p_c}\left( |K_v| \geq n\right) & \asymp n^{-1/2} &n&\nearrow \infty 
\label{exponent:volume}
\\
\bP_{p_c}\left( \operatorname{rad}_\mathrm{int}(K_v) \geq n\right) & \asymp n^{-1} &n&\nearrow \infty.
\label{exponent:intradius}
\end{align}
\endgroup
\end{corollary}

Here, we define the \textbf{susceptibility} $\chi_p(v)$ to be the expected volume of the cluster at $v$, and define $\chi_p^{(k)}(v)$ to be the $k$th moment of the volume of the cluster at $v$.  The implicit constants in \eqref{exponent:gap} may depend on $k$. We denote the cluster at $v$ by $K_v$, writing $|K_v|$ for its volume and rad$_\mathrm{int}(K_v)$ for its \textbf{intrinsic radius}, i.e., the maximum distance  between $v$ and some other point in $K_v$ as measured by the graph distance in $G[p]$. We write $\bP_p$ and $\bE_p$ for probabilities and expectations taken with respect to the law of $G[p]$. 
We remark that the susceptibility upper bound of \eqref{exponent:susceptibility} is proven as an intermediate step in the proofs of \cref{thm:pcpu,thm:triangle}.
Further applications of our techniques to the computation of critical exponents for percolation on hyperbolic graphs, including the computation of the extrinsic radius exponent, are given in the companion paper \cite{Hutp2to2}.

Finally, we remark that the following corollary of  \cref{thm:pcpu} is implied by the work of Lyons, Peres, and Schramm \cite{LPS06}. We refer the reader to that paper and to \cite[Chapter 11]{LP:book} for background on minimal spanning forests.

\begin{corollary}
\label{cor:MSF}
Let $G$ be a connected, locally finite,  nonamenable, Gromov hyperbolic, quasi-transitive graph.  
 Then the free and wired minimal spanning forests of $G$ are distinct.
\end{corollary}


\subsection{Organisation and overview}
\label{subsec:overview}

The proof of our main theorems has two parts. The first part, which is contained in \cref{sec:l2}, applies to arbitrary quasi-transitive graphs. In that part of the paper, we introduce a new critical parameter $p_{2\to 2}$, defined to be the supremal value of $p$ for which the matrix $T_p$ defined by $T_p(u,v)=\tau_p(u,v)$ is bounded as a linear operator from $L^2(V)$ to $L^2(V)$. We observe that $p_c\leq p_{2\to 2}\leq p_u$ and that $\nabla_p<\infty$ for all $p<p_{2\to 2}$, and derive a generally applicable necessary and sufficient condition for the strict inequality $p_c<p_{2\to 2}$.
In particular, we show that a quasi-transitive graph has $p_c<p_{2\to 2}$ if and only if
\vspace{-0.75em}
\begin{align}
\vspace{-1em}
\limsup_{p\uparrow p_c}\, (p_c-p)\overline{\chi}_p &< \infty
\label{eq:intro1}
\intertext{and\vspace{-0.75em}}
\vspace{-3em}
 \lim_{p \uparrow p_c} \sup\left\{ \frac{\sum_{u,v\in K}\tau_p(u,v)}{\overline{\chi}_p |K|} : K \subseteq V \text{ finite}\right\} &=0,
 \label{eq:intro2}
\end{align}
where $\overline{\chi}_p=\sup_{v\in V} \chi_p(v)$. We also prove some related results concerning the similarly defined critical parameters $p_{q\to q}$ for $q\in [1,\infty]$. Finally, we apply the results of \cite{Hutchcroftnonunimodularperc} to prove that $p_c<p_{2\to 2}$ in the nonunimodular setting.

The second part of the paper spans \cref{sec:geometry,sec:Magic,sec:mainproof}, and is specific to the Gromov hyperbolic setting. In that part of the paper, following a review of relevant background and the proof of a few preliminary geometric facts in \cref{sec:geometry}, we verify that \eqref{eq:intro1} and \eqref{eq:intro2} hold under the hypotheses of \cref{thm:pcpu,thm:triangle}. 
The starting point for this analysis was the observation by Benjamini \cite{benjamini2016self} that in any nonamenable Gromov hyperbolic graph, a constant fraction of any finite set of vertices lie near the boundary of the convex hull of the set, a fact he deduced from related work of Benjamini and Eldan \cite{MR2970060} (similar observations appeared earlier in \cite{madras2010}). 
In \cref{subsec:susceptibility}, we apply a variation on this observation to establish a differential inequality for the susceptibility which implies that \eqref{eq:intro1} holds under the hypotheses of \cref{thm:pcpu,thm:triangle}.

In \cref{sec:Magic}, we apply the so-called \emph{Magic Lemma} of Benjamini and Schramm \cite[Lemma 2.3]{BeSc}
to prove a refinement of this observation, which, roughly speaking, states that for any finite set of vertices in a Gromov hyperbolic graph, from the perspective of most vertices in the set, most of the set is contained in either one or two distant half-spaces. In
\cref{subsec:criticaldecay}, we apply this geometric fact to prove that \eqref{eq:intro2} holds under the hypotheses of \cref{thm:pcpu,thm:triangle}, completing the proof of our main theorems. To do this, we use a mixture of probabilistic and geometric techniques to show that a distant half-space can only contribute a small amount to the susceptibility, which yields \eqref{eq:intro2} when combined with the aforementioned consequence of the Magic Lemma.

Finally, in \cref{sec:closing} we conclude the paper with some remarks, conjectures, and open problems. In particular, we remark there that the proof given in \cite{1712.04911} of the estimate known as \emph{Schramm's Lemma} shows that the same estimate continues to hold at $p_{2\to 2}$, and consequently that there cannot be a unique infinite cluster at $p_{2\to 2}$.


\section{An operator-theoretic perspective on percolation}
\label{sec:l2}

In this section, we develop an `operator-theoretic perspective' on percolation. 
 In particular, we introduce a new critical parameter $p_{2\to2}$ which satisfies $p_c \leq p_{2\to2} \leq  p_u$ and $\nabla_p <\infty$ for all $p<p_{2\to2}$. This allows us to state \cref{thm:pell2}, which strengthens both \cref{thm:pcpu,thm:triangle}. We also give a sufficient condition for $p_c<p_{2\to2}$ which will be applied to Gromov hyperbolic graphs in the remainder of the paper. The approach taken in this section  was inspired in part by Gady Kozma, who advocated the application of operator-theoretic techniques to percolation in \cite{MR2779397}.

Let $G=(V,E)$ be a connected, locally finite graph, and let $\R^V$ be the space of real-valued functions on $V$. 
For each matrix $M \in [-\infty,\infty]^{V^2}$, we define $\mathscr{D}(M) \subseteq \R^V$ to be the set of $f\in \R^V$ such that $\sum_{v\in V} |f(v)| |M(u,v)| <\infty$ for every $u\in V$, so that $M$ defines a linear operator from $\mathscr{D}(M)$ to $\R^{V^2}$. 
Recall that for each $q,q'\in [1,\infty]$ the $q \to q'$ norm of $M$ is defined by $\|M\|_{q \to q'} =\infty$ if $L^q(V) \nsubseteq \mathscr{D}(M)$, and otherwise by
\[
\|M\|_{q \to q'} = \sup\left\{ \frac{\|Mf\|_{q'}}{\|f\|_q} : f \in L^q(V), f\neq 0 \right\}.
\]

Now consider the matrix $T_p$ whose entries are given by the two-point function $T_p(u,v)=\tau_p(u,v)$. Since $\tau_p(x,y)=\tau_p(y,x)$ for every $x,y\in V$, the matrix $T_p$ is symmetric and its associated operator is self-adjoint\footnote{It is an observation of Aizenman and Newman \cite{MR762034} that $T_p$ is positive definite in the sense that $\langle T_p f, f\rangle >0$ for every non-zero $f\in L^2(V)$. We shall not use this property here.}.
The $1\to 1$ and $\infty\to\infty$ norms of $T_p$ are given precisely by the susceptibility:
\[\| T_p \|_{1\to1} = \| T_p \|_{\infty\to\infty} = \sup_{v\in V} \sum_{v\in V} \tau_p(v,u) = \sup_{v\in V} \chi_p(v)=: \overline{\chi}_p.\]
It follows by sharpness of the phase transition \cite{aizenman1987sharpness,antunovic2008sharpness,duminil2015new} that if $G$ is quasi-transitive then
$\|T_p\|_{1\to1}<\infty$ if and only if $p<p_c$.
On the other hand, we can also consider the $q\to q$ norm of $T_p$ for other $q\in [1,\infty]$, and define $p_{q\to q}$ to be the critical value associated to the finiteness of $\|T_p\|_{q\to q}$, that is,
\[p_{q\to q} = p_{q\to q}(G) = \sup\Bigl\{p\in [0,1]: \|T_p\|_{q\to q} <\infty\Bigr\}.\]
Since $T_p$ is symmetric we have that 
\begin{equation}
\label{eq:conjugate}
\|T_p\|_{q\to q} = \|T_p\|_{\frac{q}{q-1}\to \frac{q}{q-1}}
\end{equation}
for every $q\in [1,\infty]$. Moreover, it follows from the Riesz-Thorin Theorem that $\log \| T_p \|_{1/q\to1/q}$ is a convex function of $q\in [0,1]$. Together, these facts imply that $\|T_p\|_{q\to q}$ is a decreasing function of $q$ on $[1,2]$ and an increasing function of $q$ on $[2,\infty]$, so that $p_{q\to q}$ is an increasing function of $q$ on $[1,2]$ and a decreasing function of $q$ on $[2,\infty]$. 
In particular, \[p_c(G)=p_{1\to 1}(G)=p_{\infty\to\infty}(G)\leq p_{q\to q}(G)\leq p_{2\to 2}(G)\] for every quasi-transitive graph $G$ and $q\in [1,\infty]$. 
We will be particularly interested in the critical value $p_{2\to 2}(G)$, which by the above discussion is equal to $\sup_{q\in [1,\infty]}p_{q\to q}(G)$. 

If $G$ is quasi-transitive and $G[p]$ has a unique infinite cluster, then the two-point function $\tau_p(u,v)\geq \mathbf{P}_p(u\to\infty)\mathbf{P}_p(v\to\infty)$ is bounded below by a positive constant, and it follows that $p_{q\to q}(G)\leq p_u(G)$ whenever $q\in [1,\infty]$ and $G$ is infinite and quasi-transitive. (In \cref{sec:closing}, we prove the stronger statement that $G[p_{2\to 2}]$ does not have a unique infinite cluster a.s.\ when $G$ is nonamenable and quasi-transitive.) Thus, the following theorem, which is proven in \cref{sec:mainproof}, strengthens \cref{thm:pcpu}. The difficult part of the theorem is that $p_c(G)<p_{2\to 2}(G)$; we shall see in \cref{prop:2to2givesqtoq} that this implies that $p_c(G)<p_{q\to q}(G)$ for every $q \in (1,\infty)$. The dependence of $p_{q\to q}$ on $q$ is further investigated in \cite{Hutp2to2}.

\begin{thm}
\label{thm:pell2}
Let $G$ be a connected, locally finite, quasi-transitive, nonamenable, Gromov hyperbolic graph. Then $p_c(G)<p_{q\to q}(G)$ for every $q\in (1,\infty)$. 
\end{thm}


We now briefly discuss the relationship between the $2\to 2$ norm and diagramatic sums. 
Recall that the \textbf{$n$th polygon diagram} at $v$ is defined to be
\begin{equation*}
T_p^n(v,v)
= \langle T_p^n \mathbbm{1}_v,\mathbbm{1}_v\rangle= 
\sum_{x_1,x_2,\ldots,x_{n-1}\in V} \tau_p(v,x_1)\tau_p(x_1,x_2)\cdots \tau_p(x_{n-2},x_{n-1})\tau_p(x_{n-1},v),
\end{equation*}
so that $\nabla_p= \sup_{v\in V} T^3_p(v,v)$. (Note that $T_p^n$ is always well-defined as an element of $[0,\infty]^{V^2}$.) 
It follows by the Cauchy-Schwarz inequality and the symmetry of $T_p$ that
\begin{equation}
\label{eq:diagrambound}
T_p^n(v,v)\leq \|T_p^n\|_{2\to 2} = \|T_p\|_{2\to 2}^n
\end{equation}
for every $v\in V$ and $n\geq 1$, so that in particular $\nabla_p<\infty$ for all $p<p_{2\to 2}$. 
The next proposition implies that $\|T_p\|_{2\to2}$ is in fact \emph{equal} to the exponential growth rate of the polygon diagrams.

\begin{prop}
\label{lem:diagrams}
Let $M \in [0,\infty]^{V^2}$ be a non-negative symmetric matrix. Then
\[
\|M\|_{2\to2}=\sup_{v\in V, n\geq 0} \left[
M^n(v,v)\right]^{1/n}.
\]
\end{prop}

\begin{proof}
This is presumably a standard fact. 
It follows by the same proof as that of \cite[Proposition 6.6]{LP:book}, where the claim is stated in the special case that $M$ is stochastic.
\end{proof}

Our next goal is to prove the following; see \cite{Hutp2to2} for a sharp quantitative version.

\begin{prop}
\label{prop:2to2givesqtoq}
Let $G$ be a connected, locally finite, quasi-transitive graph. If $p_c(G) < p_{2\to 2}(G)$ then $p_c(G) <p_{q\to q}(G)$ for every $q\in (1,\infty)$.
\end{prop}

\cref{prop:2to2givesqtoq} will follow from a few simple lemmas, which will also be used in the proof of our criterion for $p_c<p_{2\to 2}$. The first two, \cref{lem:operatorbound,cor:l2AizBar}, follow by similar reasoning to that used in \cite{aizenman1987sharpness}, where similar inequalities are established for the susceptibility. See also \cite[\S 3]{Hutchcroftnonunimodularperc}.

Given two matrices $S,T\in [-\infty,\infty]^{V^2}$, we write $S \opleq T$ to mean that $S(u,v) \opleq T(u,v)$ for every $u,v \in V$.
It can easily be checked from the definitions that if $S,T\in [0,\infty]^{V^2}$ are non-negative matrices with $S \opleq T$ then $\|S\|_{q\to q}\leq \|T\|_{q\to q}$ for every $q\in [1,\infty]$.
Let $E^\rightarrow$ be the set of oriented edges of $G$. An oriented edge $e$ has head $e^+$ and tail $e^-$. Let $A$ be the adjacency matrix of $G$, defined by letting $A(u,v)=A(v,u)$ be the number of oriented edges with tail $u$ and head $v$.

\begin{lemma}
\label{lem:operatorbound}
Let $G$ be a connected, locally finite  graph. Then 
\[T_{p_1} \opleq T_{p_2}
\opleq 
 \sum_{k\geq 0} \left[\frac{p_2-p_1}{1-p_1}  T_{p_1} A\right]^k T_{p_1}\] for every $0\leq p_1 \leq p_2 \leq 1$.
\end{lemma}

\begin{remark}
With a little more care the $1/(1-p_1)$ factor can be removed from the bracketed expression. This yields mild improvements to \cref{cor:l2AizBar,prop:criterion} below.
\end{remark}

We briefly recall some background on correlation inequalities for percolation, referring the reader to \cite[\S 2.2 and \S 2.3]{grimmett2010percolation} for more detail. An event $\sA \subseteq \{0,1\}^E$ is said to be \textbf{increasing} if its indicator function is an increasing function of each bit. The \textbf{Harris-FKG inequality} states that increasing events are positively correlated under the product measure, that is,
\[\bP_p(\sA \cap \sB) \geq \bP_p(\sA)\bP_p(\sB)\]
for every $p\in [0,1]$ and every two increasing events $\sA,\sB \subseteq \{0,1\}^E$.
Given an increasing event $\sA$ and a configuration $\omega \in \sA$, a \textbf{witness} for $\sA$ in $\omega$ is defined to be a set $W \subseteq \{e\in E : \omega(e)=1\}$ such that $\mathbbm{1}(W) \in \sA$. For example, an open path connecting $u$ to $v$ is a witness for the event $\{u\leftrightarrow v\}$ that $u$ and $v$ are connected in $G[p]$. Given two increasing events $\sA$ and $\sB$, the \textbf{disjoint occurrence} $\sA\circ \sB$ of $\sA$ and $\sB$ is defined to be the event that there exist disjoint witnesses for $\sA$ and $\sB$. The van den Berg and Kesten inequality (a.k.a.\ the \textbf{BK inequality}) states that 
\[
\mathbf{P}_p(\sA\circ \sB) \leq \mathbf{P}_p(\sA)\mathbf{P}_p(\sB)
\]
for any increasing events $\sA,\sB$ depending on at most finitely many edges. In fact, the BK inequality applies to arbitrary product measures and does not require all edges to have the same probability of being open. It is usually unproblematic to apply the BK inequality to events depending on on infinitely many edges. For example, if $\sA$ and $\sB$ are increasing events for which every witness must have a finite subset that is also a witness (e.g., connection events) then we have that $\mathbf{P}_p(\sA \circ \sB) \leq \mathbf{P}_p(\sA)\mathbf{P}_p(\sB)$ by an obvious limiting argument; this applies every time we use the BK inequality in this paper. 

\begin{proof}[Proof of \cref{lem:operatorbound}]
The lower bound is trivial. The upper bound follows by an argument very similar to that used in e.g.\ the proof of \cite[Proposition 1.12]{Hutchcroftnonunimodularperc}, which we include for completeness.

First sample $G[p_1]$. Independently, for each edge of $G$, add an additional \textbf{blue} edge in parallel to that edge with probability $(p_2-p_1)/(1-p_1)$.   Thus, the subset of edges of $G$ that are either open in $G[p_1]$ or have a blue edge added in parallel to them is equal in distribution to $G[p_2]$. Denote the graph obtained by adding each of these blue edges to $G[p_1]$ by $\tilde G$, so that $\tau_{p_2}(u,v)$ is equal to the probability that $u$ and $v$ are connected in $\tilde G$.  Let $\tilde \tau_i(u,v)$ be the probability that $u$ and $v$ are connected by a simple path in $\tilde G$ containing exactly $i$ blue edges, and let $\tilde T_i\in[0,\infty]^{V^2}$ be the matrix defined by $\tilde T_i (u,v)= \tilde \tau_i(u,v)$. Then $\tilde T_0 =T_{p_1}$, 
\[\tau_{p_2}(u,v)\leq \sum_{i\geq0}\tilde\tau_i(u,v) \qquad \text{ and }\qquad T_{p_2} \opleq \sum_{i\geq0} \tilde T_i.\]
Considering the location of the last blue edge used in a simple path from $u$ to $v$ in $\tilde G$ and applying the BK inequality yields that
\[
\tilde \tau_{i+1}(u,v) \leq \frac{p_2-p_1}{1-p_1}\sum_{w \in V} \tilde \tau_i(u,w) \sum_{e^- =w} \tau_{p_1}(e^+,v),
\]
which is equivalent to the inequality
\[
\tilde T_{i+1} \opleq \frac{p_2-p_1}{1-p_1}T_{p_1} A \tilde T_i.
\]
Inducting over $i$ yields that $\tilde T_i \opleq \left[\frac{p_2-p_1}{1-p_1} T_{p_1} A \right]^i T_{p_1}$ for every $i\geq 0$, and the claim follows.
\end{proof}

\begin{corollary}
\label{cor:l2AizBar}
Let $G$ be an infinite, connected, locally finite  graph. Then 
\[\|T_p\|_{q\to q} \geq  \frac{1-p}{\|A\|_{q\to q}(p_{q\to q}-p)}  \]
for every $0\leq p<p_{q\to q}$. In particular, $\|T_{p_{q\to q}}\|_{q\to q}=\infty$.
\end{corollary}

Note that $\|A\|_{q\to q} \leq \|A\|_{1\to 1}$ is bounded by the maximum degree of $G$. 

\begin{proof}
If $G$ has unbounded degrees then $p_{q\to q}(G)=0$ for every $q\in [1,\infty]$ and the claim is trivial, so suppose not. 
Applying \cref{lem:operatorbound} we have that
\begin{align*}
\|T_{p'}\|_{q\to q} 
&\leq \biggl\|\sum_{k\geq0} \left[\frac{p'-p}{1-p} AT_p\right]^kT_p\biggr\|_{q\to q} \leq \sum_{k\geq0} \left(\frac{p'-p}{1-p}\right)^k\|A\|_{q \to q}^k\| T_p\|_{q \to q}^{k+1}
\end{align*}
for every $0\leq p \leq p' \leq 1$.
We deduce immediately that the set $\{p\in [0,1]:\|T_{p_{q\to q}}\|_{q \to q}<\infty\}$ is open in $[0,1]$, and consequently that $\| T_{p_{q \to q}}\|_{q\to q}=\infty$ (the assumption that $G$ is infinite handles the degenerate case $p_{q\to q}=1$). Taking $0\leq p<p_{q\to q}$ and $p'=p_{q \to q}$, we deduce that
\[
\sum_{k\geq0} \left(\frac{p_{q\to q}-p}{1-p}\right)^k\|A\|_{q \to q}^k\| T_p\|_{q \to q}^{k+1} = \infty,
\]
which implies the desired inequality.
\end{proof}



\begin{proof}[Proof of \cref{prop:2to2givesqtoq}] 
By \eqref{eq:conjugate} it suffices to prove the claim for $q\in (1,2)$. 
Since $p_c(G)<p_{2\to 2}(G)$, we have that $\nabla_{p_c}<\infty$ by \eqref{eq:diagrambound}, and hence by the results of \cite{MR762034} and \cite[\S 7]{Hutchcroftnonunimodularperc} that there exists a constant $C$ such that $\|T_p\|_{1\to 1}=\overline{\chi}_p \leq C(p_c-p)^{-1}$ for all $0\leq p <p_c$ . Let $q\in (1,2)$ and let $\theta\in (0,1)$ be such that $1/q= (1-\theta) + \theta/2$. Then we have by the Riesz-Thorin Theorem that
\[
\liminf_{p\uparrow p_c} (p_c-p)\|T_p\|_{q\to q} \leq 
\liminf_{p\uparrow p_c} (p_c-p)\|T_p\|_{1\to 1}^{1-\theta} \|T_p\|_{2 \to 2}^\theta \leq
\liminf_{p\uparrow p_c} C^{1-\theta}(p_c-p)^\theta \|T_{p_c}\|_{2\to 2}^\theta =0,
\]
and it follows from \cref{cor:l2AizBar} that $p_c(G)<p_{q\to q}(G)$ as claimed.
\end{proof}


We next state our sufficient condition for $p_c<p_{2\to2}$.
For each $0\leq p<p_c$, we define 
\[
\iota(T_p) = 1 - \sup\left\{ \frac{\sum_{u,v\in K} \tau_{p}(u,v)}{\overline{\chi}_{p} |K|} : K \subseteq V \text{ finite} \right\}.
\]
We interpret $\iota(T_p)$ as an isoperimetric constant that measures the extent to which percolation clusters are inclined to escape fixed finite sets: It is the Cheeger constant of the symmetric matrix $\overline{\chi}_p^{-1} T_p$, which is stochastic\footnote{Recall that a matrix is stochastic if it has non-negative entries and all of its row sums are $1$, i.e., if it is the transition matrix of some Markov chain. A matrix is substochastic if it has non-negative entries and its row sums are at most $1$, i.e., if it is the transition matrix of a Markov chain with killing. In the transitive case, the random walk associated to the stochastic matrix $\overline{\chi}_p^{-1} T_p$ can be interpreted as follows: At each step of the walk, we sample a size-biased percolation cluster at the vertex the walk currently occupies, independently of all previous steps, and then move to a vertex chosen uniformly at random from within this cluster. The quasi-transitive case is similar except that, at each step, the walk is killed with probability $(\overline{\chi}_p-\chi_p(v))/\overline{\chi}_p$ when it is at the vertex $v$.} when $G$ is transitive and is substochastic when $G$ is quasi-transitive.

\begin{prop}
\label{prop:criterion}
Let $G$ be a connected, locally finite, quasi-transitive graph. Then $p_c(G)<p_{2\to2}(G)$ if and only if
\begin{equation}
\label{eq:sufficientcondition}
\liminf_{p\uparrow p_c} \frac{p_c-p}{1-p} \overline{\chi}_{p} \sqrt{1-\iota(T_p)^2} < \frac{1}{\|A\|_{2\to2}}. 
\end{equation}
In particular, if \eqref{eq:sufficientcondition} holds then $p_c(G)<p_u(G)$ and $\nabla_{p_c}<\infty$.
\end{prop}

 We will apply \cref{prop:criterion} by showing that the limit infimum in question is equal to zero under the hypotheses of \cref{thm:pell2}. 
Note that if $G$ is transitive with vertex degree $k$ then $\|A\|_{2\to2}= k\rho(G)$, where $\rho(G)$ is the spectral radius of the random walk on $G$.

\begin{lemma} Let $G$ be a connected, locally finite  graph. Then
\label{lem:Cheeger}
\[
\overline{\chi}_p \left(1-\iota(T_p)\right) \leq \|T_p\|_{2\to2} \leq \overline{\chi}_p \sqrt{1-\iota(T_p)^2}
\]
for every $0<p<p_{1\to 1}(G)$.
\end{lemma}

\begin{proof}
%
The normalized matrix $\hat T_p := \overline{\chi}_p^{-1}T_p$ is symmetric and substochastic, and $\iota(T_p)$ is its Cheeger constant. Thus, the claim follows from Cheeger's inequality, see e.g.\ \cite[Theorem 6.7]{LP:book}. (Cheeger's inequality is usually stated for self-adjoint Markov operators but the proof is valid for self-adjoint sub-Markov operators, or, equivalently, symmetric substochastic matrices.)
\end{proof}

\begin{proof}[Proof of \cref{prop:criterion}]
The `if' implication is immediate from \cref{lem:Cheeger,cor:l2AizBar}. For the `only if' implication, note that if $p_c<p_{2\to2}$ then $\nabla_{p_c}<\infty$ by \eqref{eq:diagrambound}, and consequently that there exists a constant $C$ such that $\overline{\chi}_{p} \leq C(p_c-p)^{-1}$ for all $p<p_c$, as discussed in the proof of \cref{prop:2to2givesqtoq}. On the other hand, by \cref{lem:Cheeger} we must have that $\iota(T_p)\to1$ as $p\to p_c$, and the claim follows.
\end{proof}

Next, we prove that the following theorem can be deduced immediately from the results of \cite{Hutchcroftnonunimodularperc} and \cref{lem:diagrams}, so that it suffices for us to prove \cref{thm:pell2} in the unimodular case.

\begin{thm}
\label{thm:nonunimodular}
Let $G$ be a connected, locally finite graph, and suppose that $\Aut(G)$ has a quasi-transitive nonunimodular subgroup. Then $p_c(G)< p_{q \to q}(G)$ for every $q\in (1,\infty)$.
\end{thm}

\begin{proof}
By \cref{prop:2to2givesqtoq}, it suffices to prove that $p_c(G)<p_{2\to 2}(G)$. 
We use the notation of \cite{Hutchcroftnonunimodularperc}. Let $\Gamma$ be a quasi-transitive nonunimodular subgroup of $\Aut(G)$.
It follows from 
the proof of \cite[Lemma 7.1]{Hutchcroftnonunimodularperc} that 
\[\sup_{v\in V}T_p^n(v,v) \leq \left(\overline{\chi}_{p,\lambda}\right)^n\] for every $p\in[0,1]$ and $\lambda \in \R$, and hence by \cref{lem:diagrams} that $\|T_p\|_{2\to2} \leq \overline{\chi}_{p,\lambda}$. It follows that $p_{2\to2}(G) \geq p_t(G,\Gamma)$, and the claim follows from \cite[Theorem 1.11]{Hutchcroftnonunimodularperc}.
\end{proof}

\section{Geometric preliminaries}
\label{sec:geometry}

We now move away from the generalities of the previous section, and from now on will restrict attention to the Gromov hyperbolic setting. 
In this section, we provide geometric background on Gromov hyperbolic graphs that will be applied to study percolation in \cref{sec:Magic,sec:mainproof}. For more detailed background on various aspects of hyperbolic geometry, see e.g.\ \cite{anderson2006hyperbolic,Gromov87,ghys1990groupes,MR2243589}. An overview of Gromov hyperbolicity particularly accessible to probabilists is given in \cite[\S 20B and \S22]{Woess}.

\subsection{Hyperbolic space}

\label{subsec:Hdbackground}

Recall that for $d\geq 2$, the \textbf{hyperbolic $d$-space} $\mathbb{H}^d$ is defined to be the unique complete, connected, simply connected, $d$-dimensional Riemannian manifold of constant sectional curvature $-1$. Hyperbolic $1$-space $\mathbb{H}^1$ is defined to be isometric to $\R$. Throughout this paper, we employ the \textbf{Poincar\'e half-space model} to  identify $\mathbb{H}^d$ with the open half-space $\R^d_+ = \R^{d-1}\times (0,\infty)$ 
 equipped with the Riemannian metric given by the length element
\[
ds = \frac{1}{x_d}\sqrt{dx_1^2+\cdots+dx_d^2}.
\]
This Riemannian metric is given explicitly by
\[
d_\mathbb{H}\left((x_1,y_1),(x_2,y_2)\right) = 2 \log \frac{\sqrt{\|x_1-x_2\|_2^2+(y_1-y_2)^2} + \sqrt{\|x_1-x_2\|_2^2+(y_1+y_2)^2} }{2\sqrt{y_1y_2}}
\]
for every $x_1,x_2\in\R^{d-1}$ and $y_1,y_2\in (0,\infty)$. 
Geodesics in $\mathbb{H}^d$ correspond to circles and lines in $\R^d$ that are orthogonal to the boundary hyperplane $\R^{d-1}$. The following operations all induce isometries of $\mathbb{H}^d$ under this identification:
Dilations of $\R^d$, isometries of $\R^d$ fixing $\R^{d-1}$, and inversions of $\R^d \cup \{\infty\}$ through Euclidean spheres and hyperplanes that are orthogonal to $\R^{d-1}$ (equivalently, reflections through hyperplanes in $\mathbb{H}^d$). In particular, it follows that for any two points $x,y \in \mathbb{H}^d$, there exists an isometry $\gamma$ of $\mathbb{H}^d$ such that $\gamma x = (0,\ldots,0,1)$ and $\gamma y = (0,\ldots,0,\exp d_\mathbb{H}(x,y))$.

Recall that a \textbf{half-space} in $\mathbb{H}^d$ is a set of the form $H=H(a,b):=\{x\in \mathbb{H}^d : d(x,a) \leq d(x,b) \}$ for some $a \neq b\in \mathbb{H}^d$.
 The topological boundary $\partial H$ of a half-space $H$ is called a \textbf{hyperplane}, so that a subset of $\mathbb{H}^d$ is a hyperplane if and only if it is equal to $ \partial H(a,b):=\{x\in \mathbb{H}^d : d(x,a) = d(x,b) \}$ for some $a,b \in \mathbb{H}^d$. Hyperplanes in $\mathbb{H}^d$ are isometric to $\mathbb{H}^{d-1}$ when equipped with the induced metric. In the Poincar\'e half-space model, hyperbolic hyperplanes are represented by Euclidean spheres and hyperplanes that are orthogonal to the boundary $\R^{d-1}$. Note that if $a\in \mathbb{H}^d$ and $H\subseteq \mathbb{H}^d$ is a half-space containing $a$, then there exists a unique $b \in \mathbb{H}^d$ such that $H=H(a,b)$, which is obtained by reflecting $a$ through the hyperplane $\partial H$.

 Note that if $H(a,b) \subseteq \mathbb{H}^d$ is a half-space with $d(a,b)> r$ and $c$ is the point on the geodesic between $a$ and $b$ that has $d(a,c)=r$, then the 
set $\bigl\{x\in \mathbb{H}^d : d(x,a)\leq d(x,b)+r \bigr\}$
is not itself a half-space, but is contained in the half-space $H(c,b)$. This set also contains the $r/2$-neighbourhood of $H(a,b)$, so that
\begin{equation}
\label{eq:perturbedhalfspace1}
\{x \in \mathbb{H}^d : d(x,H(a,b)) \leq r/2 \} \subseteq \left\{x \in \mathbb{H}^d : d(x,a)\leq d(x,b)+r \right\}\subseteq H(c,b).
\end{equation}
Similarly, if $d$ is the point that lies on the infinite extension of the geodesic from $b$ to $a$ and has distance $r$ from $a$ and $d(a,b)+r$ from $b$, then we have that
 \begin{equation}
 \label{eq:perturbedhalfspace2}
 \{x \in \mathbb{H}^d : d(x,H(b,a)) \geq r/2 \} \supseteq \left\{x\in 
\mathbb{H}^d : d(x,a)\leq d(x,b)-r \right\}
 \supseteq H(d,b).
 \end{equation}
  Both claims can easily be verified by trigonometric calculations.


\subsection{Hyperbolic graphs}
Let $G$ be a graph, and let $x,y,w$ be vertices of $G$. The \textbf{Gromov product} $(x\mid y)_w$ is defined to be
\[
(x\mid y)_w = \frac{1}{2}\left(d(w,x)+d(w,y)-d(x,y)\right).
\]
 Let $\delta\geq 0$. We say that $G$ is $\delta$\textbf{-hyperbolic} if the inequality 
\[
(x\mid z)_w \geq (x\mid y)_w \wedge (y\mid z)_w - \delta
\]
holds for every $w,x,y,z\in V$, and that $G$ is \textbf{Gromov hyperbolic} if it is $\delta$-hyperbolic for some $\delta \geq 0$. (Note that while the use of the letter $\delta$ to describe this parameter is traditional, it should not generally be thought of as being small.)
Roughly speaking, in a Gromov hyperbolic graph, the Gromov product $(x \mid y)_w$ measures the distance from $w$ at which the geodesics from $w$ to $x$ and from $w$ to $y$ begin to diverge from each other.
As mentioned in the introduction, Gromov hyperbolicity can be defined equivalently by the Rips thin triangle property.


\subsection{The Bonk-Schramm Theorem}

\label{subsec:BonkSchramm}

The following theorem of Bonk and Schramm \cite{MR1771428} relates the abstract notion of Gromov hyperbolicity with the geometry of the concrete spaces $\mathbb{H}^d$. It will be the main tool by which we reason about the geometry of hyperbolic graphs in this paper. 

We must first introduce some definitions. 
Let  $f:(X,d_X)\to (Y,d_Y)$ be a function between metric spaces. Given $k\in [0,\infty)$, we say that $f$ is $k$\textbf{-cobounded} if $d_Y(y, f(X))\leq k$ for every $y\in Y$. Given $k\in [0,\infty)$ and $\lambda \in (0,\infty)$, we say that $f$ is a $(\lambda,k)$\textbf{-rough similarity} if it is $k$-cobounded and
\[
\bigl| \lambda d_X(x_1,x_2) - d_Y(f(x_1),f(x_2)) \bigr|\leq k
\]
for every $x_1,x_2\in X$. Note that this is a much stronger condition than being a rough isometry in the usual sense, and is particularly useful when discussing half-spaces.

Note that every closed convex subset $X \subseteq \mathbb{H}^d$ is a Gromov hyperbolic, geodesic metric space when equipped with the restriction of the metric on $\mathbb{H}^d$.

\begin{thm}[Bonk and Schramm]
Let $G$ be a bounded degree, connected, Gromov hyperbolic graph. Then there exists $d\geq 1$ such that $G$ is roughly similar to a closed convex subset $X\subseteq \mathbb{H}^d$. 
\end{thm}

\subsection{The Gromov boundary}

We now review the definition of the Gromov boundary. Proofs of the facts in this section can be found in \cite{MR1771428} and references therein. 
Let $G$ be a connected, locally finite, Gromov hyperbolic graph, and let $w$ be a fixed vertex. A sequence of vertices $\langle v_i \rangle_{i\geq }$ in $G$ is said to \textbf{converge at infinity} if 
$\lim_{n\to\infty}\inf_{i,j\geq n}(v_i\mid v_j)_w=\infty$.
If $\langle u_i\rangle_{i\geq 1}$ and $\langle v_i \rangle_{i\geq }$ are two sequences of vertices that both converge at infinity, we say that the two sequences are \textbf{equivalent} if $(u_i\mid v_i)_w \to \infty$ as $i\to\infty$. This defines an equivalence relation on the set of sequences that converge at infinity. We define the \textbf{Gromov boundary} $\delta G$ of $G$ to be the set of equivalence classes of this equivalence relation. Neither convergence at infinity or equivalence of convergent sequences depends on the choice of $w$. Given $\xi \in \delta G$ and a sequence $\langle v_i \rangle_{i\geq1}$ in $V$, we write $v_i\to \xi$ if $\langle v_i \rangle_{i\geq1}$ converges at infinity and is an element of the equivalence class $\xi$.
 Given two points $\xi,\zeta\in \delta G$ and $v\in V$ we define 
\[
(\xi \mid \zeta)_w := \sup \Bigl\{ \liminf_{i\to\infty} ( \xi_i \mid \zeta_i) : \xi_i\to \xi,\, \zeta_i\to \zeta\Bigr\} \quad \text{and} \quad
(\xi \mid v)_w := \sup \Bigl\{ \liminf_{i\to\infty} ( \xi_i \mid v) : \xi_i\to \xi\Bigr\}.
\]

We now define the topology on $\delta G$. 
If $\eps>0$ is sufficiently small, then there exists a metric $d_{w,\eps}$ on $V\cup\delta G$ with the property that
\[
\frac{1}{2}e^{-\eps(x\mid y)_w} \leq d_w(x,y) \leq e^{-\eps (x \mid y)_w}
\] 
for every $x,y\in V \cup \delta G$. We equip $V \cup \delta G$ with the topology induced by this metric, which does not depend on the choice of $w$ or $\eps$ provided that $\eps$ is sufficiently small. The resulting topological space is compact \cite[Corollary 22.13]{Woess}.

The definition of the Gromov boundary given above extends naturally to any Gromov hyperbolic metric space. 
If we identify $\mathbb{H}^d$ with the half-space $\R^d_+$ via the Poincar\'e half-space model, then the Gromov boundary $\delta \mathbb{H}^d$ of $\mathbb{H}^d$ can be identified with $\R^{d-1}\cup\{\infty\}$. 
More generally, if $X \subseteq \mathbb{H}^d$ is closed and convex, then the Gromov boundary $\delta X$ of $X$ can be identified with the intersection of the closure of $X$ in $\R^d_{\geq 0} \cup \{\infty\}$ with $\R^{d-1} \cup \{\infty\}$.
If $G$ has bounded degrees and $\phi:V\to X$ is a rough similarity between $G$ and a convex set $X\subseteq \mathbb{H}^d$, then $\phi$ extends to a unique continuous function $\hat \phi:V\cup\delta G \to X \cup \delta X$, and the restriction $\delta \phi$ of $\hat \phi$ to $\delta G$ is a homeomorphism $\delta \phi:\delta G\to\delta X$. 

We say that $G$ is \textbf{visible from infinity} if there exists a constant $C$ such that every vertex of $G$ is at distance at most $C$ from some doubly-infinite geodesic of $G$. Every nonamenable Gromov hyperbolic graph is visible from infinity, as is every infinite, quasi-transitive Gromov hyperbolic graph. If $G$ is visible from infinity, then the space $X$ in the Bonk-Schramm theorem is also visible from infinity (with the obvious extension of the definition), and is therefore easily seen to be roughly similar (with $\lambda=1$) to the hyperbolic convex hull of its boundary $\delta X \subseteq \R^{d-1}\cup\{\infty\}$. Thus, if $G$ is visible from infinity, we can take the space $X$ in the Bonk-Schramm Theorem to be equal to the convex hull of its boundary $\delta X \subseteq \R^{d-1}\cup\{\infty\}$.

\subsection{The action of automorphisms on the boundary}

Now suppose that $G=(V,E)$ is a Gromov hyperbolic graph, and consider the automorphism group $\Aut(G)$ of $G$. The the action of $\Aut(G)$ on $V$ extends uniquely to a continuous action of $\Aut(G)$ on $V\cup\delta G$ \cite[Theorem 22.14]{Woess}. The elements of $\Aut(G)$ can be classified as \emph{elliptic, parabolic, and hyperbolic} as follows.
\begin{enumerate}
\item We say that $\gamma\in \Aut(G)$ is \textbf{elliptic} if it fixes some finite set of vertices $K \subseteq V$.
\item We say that $\gamma\in \Aut(G)$ is \textbf{parabolic} if it fixes a unique boundary point $\xi\in \delta G$ and
\[\lim_{n\to\infty}\gamma^n x \to \xi \qquad \text{ and } \qquad \lim_{n\to\infty}\gamma^{-n} x \to \xi\]
for every $x \in V \cup \delta G$, uniformly on compact subsets of $V \cup \delta G \setminus \{\xi\}$.
\item We say that $\gamma \in \Aut(G)$ is \textbf{hyperbolic} if it fixes exactly two points $\xi,\eta$ of $\delta G$ and 
\[\lim_{n\to\infty}\gamma^n x \to \xi \qquad \text{ and } \qquad \lim_{n\to\infty}\gamma^{-n} x \to \eta\]
for every $x \in V \cup \delta G \setminus \{\xi,\eta\}$, uniformly on compact subsets of $V \cup \delta G \setminus \{\eta\}$ and $V \cup \delta G \setminus \{\xi\}$ respectively. We call $\xi$ and $\eta$ the \textbf{forward} and \textbf{backward} fixed points of $\gamma$ respectively.
\end{enumerate}

It is clear that these classes are mutually exclusive, and in fact we have the following. 

\begin{prop}
\label{prop:structure}
Let $G$ be a Gromov hyperbolic graph. Then every $\gamma \in \Aut(G)$ is either elliptic, parabolic, or hyperbolic.
\end{prop}

\cref{prop:structure} is a direct analogue of the corresponding statement for isometries of $\mathbb{H}^d$, which is classical. See \cite[\S 4]{MR1921706} and references therein for a proof in the case that $G$ is a Cayley graph of the hyperbolic group $\Gamma$, and \cite[Theorem 1 and Corollary 4]{MR1245204} for a proof in full generality.

The following proposition extends this structure theory to \emph{groups} of automorphisms. Let $G$ be a Gromov hyperbolic graph and let $\Gamma$ be a subgroup of $\Aut(G)$. The \textbf{limit set} $L(\Gamma)$ of $\Gamma$ is defined to be the set of accumulation points in $\delta G$ of the orbit $\{\gamma v : \gamma \in \Gamma\}$ for some vertex $v$ of $G$. The set of accumulation points does not depend on the choice of $v$. If $\Gamma$ is quasi-transitive then $L(\Gamma)=\delta G$, and if $G$ is infinite and quasi-transitive then  $|\delta G| \in \{2,\infty\}$.  
A pair of boundary points $(\xi,\eta) \in \delta G^2$ is said to be  a \textbf{pole pair} of $\Gamma$ if there exists a hyperbolic element $\gamma \in \Gamma$ with forward fixed point $\xi$ and backward fixed point $\eta$.

\begin{prop}
\label{prop:density}
Let $G$ be a Gromov hyperbolic graph, and let $\Gamma$ be a subgroup of $\Aut(G)$. Then precisely one of the following holds:
\begin{enumerate}
	\item[\emph{(a)}] $L(\Gamma)=\emptyset$, and $\Gamma$ fixes some finite non-empty set of vertices in $G$.
\item[\emph{(b)}] 
$|L(\Gamma)|=2$, $\Gamma$ fixes a unique pair of elements of $\delta G$, and $L(\Gamma)$ is equal to this pair.
\item[\emph{(c)}] 
$|L(\Gamma)|\in \{1,\infty\}$, $\Gamma$ fixes a unique element $\xi_0$ of $\delta G$, and does not fix any compact non-empty proper subset of $L(\Gamma) \setminus \{\xi_0\}$.
\item[\emph{(d)}] $|L(\Gamma)| = \infty$, $\Gamma$ does not fix any compact non-empty proper subset of $L(\Gamma)$, and the set of pole pairs of $\Gamma$ is dense in $L(\Gamma)^2$.
\end{enumerate}
\end{prop}

$\Gamma$ is said to have the \textbf{fixed set property}  if one of (a), (b), or (c) holds above. 



See \cite[\S 4]{MR1921706} and references therein for a proof of \cref{prop:density} in the case that $G$ is a Cayley graph of the hyperbolic group $\Gamma$. At the stated level of generality,  Woess established everything other than the density of pole pairs in the case that $\Gamma$ does not have the fixed set property; see \cite[Theorems 2 and 3 and Corollary 4]{MR1245204} and \cite[Proposition 20.10]{Woess}. A proof that the pole pairs are dense when $\Gamma$ does not have the fixed set property is given in \cite[Theorem 2.9]{Hamann2017}.

\subsection{Unimodularity}

\label{subsec:unimodularity}
We now briefly introduce unimodularity and the mass-transport principle, referring the reader to  \cite[Chapter 8]{LP:book} and \cite[\S 2.1]{Hutchcroftnonunimodularperc} for further background.

Let $G=(V,E)$ be a connected, locally finite graph, and let $\Gamma$ be a subgroup of $\Aut(G)$. We say that $\Gamma$ is \textbf{unimodular} if $|\stab_u v|= |\stab_v u|$ for every $u,v \in V$ in the same orbit of $\Gamma$, where $\stab_u = \{ \gamma \in \Gamma : \gamma u = u\}$ is the stabilizer of $u$ in $\Gamma$ and $\stab_u v = \{ \gamma v : \gamma \in \stab_u \}$ is the orbit of $v$ under $\stab_u$. We say that $G$ is unimodular if $\Aut(G)$ is unimodular. Every Cayley graph of a finitely generated group is unimodular.

If $\Gamma$ is unimodular and quasi-transitive then it satisfies the \textbf{mass-transport principle}. 
Given a graph $G$ and a subgroup $\Gamma$ of $\Aut(G)$, we write $[v]$ for the orbit of $v$ under $\Gamma$, and let $\cO$ be a set of orbit representatives of $\Gamma$, i.e., a subset of $V$ with the property that for each $u\in V$ there exists a unique $v\in \cO$ with $[u]=[v]$. The mass-transport principle  states that there exists a unique probability measure $\mu$ on $\cO$ such that whenever $\rho$ is a random root vertex on $G$ with law $\mu$ and $F:V^2\to[0,\infty]$ is invariant under the diagonal action of $\Gamma$ in the sense that $F(\gamma u, \gamma v)=F(u,v)$ for every $u,v \in V$, then
\[
\E\left[\sum_{v\in V} F(\rho,v)\right]= \E\left[\sum_{v\in V} F(v,\rho)\right].
\]
The measure $\mu$ clearly assigns a positive mass to each element of $\cO$.
For the remainder of the paper, we will use 
$\P$ and $\E$ to denote probabilities and expectations taken with respect to the law of a random root $\rho$ drawn from this measure, use 
$\P_p$ and $\E_p$ to denote probabilities and expectations taken with respect to the joint law of the random root $\rho$ and an independent percolation configuration $G[p]$, and use $\mathbf{P}_p$ and $\mathbf{E}_p$ to denote probabilities and expectations taken with respect to the marginal law of the percolation configuration $G[p]$.

A further important consequence of unimodularity is the following, which follows from \cite[Proposition 22.16]{Woess} and \cref{prop:density}.

\begin{prop}
\label{prop:unimodularnofixed}
Let $G$ be a Gromov hyperbolic graph and let $\Gamma$ be a unimodular quasi-transitive subgroup of $\Aut(G)$. Then either
\begin{enumerate}
\item $G$ is amenable and $|\delta G|=2$, or
\item $G$ is nonamenable, $|\delta G| =\infty$, $\Gamma$ does not have the fixed set property, and its set of pole pairs is dense in $\delta G^2$.
\end{enumerate}
\end{prop}

In fact, it is very unusual for $\Aut(G)$ to have the fixed set property when $G$ is Gromov hyperbolic and quasi-transitive: This can occur only if $G$ is rough-isometric to a tree \cite{MR2948665,MR3420526}.

\subsection{Discrete half-spaces in hyperbolic graphs}

Let $G=(V,E)$ be a connected, locally finite, Gromov hyperbolic graph. We say that a subset $H\subseteq V$ is a \textbf{discrete half-space} if it is of the form $H=H_G(a,b)=\{v\in V : d(v,a)\leq d(v,b)\}$ for some $a\neq b\in V$.
We say that a discrete half-space $H_G(a,b)$ is \textbf{proper} if there exist disjoint, non-empty open subsets $U_1$ and $U_2$ of $\delta G$ such that $U_1$ is disjoint from the closure of $H_G(a,b)$ in $V \cup\delta G$ and $U_2$ is disjoint from the closure of $H_G(b,a)$ in $V \cup \delta G$. See \cref{fig:proper} for examples of proper and non-proper discrete half-spaces.

\begin{figure}
\centering
\includegraphics[width=0.8\textwidth]{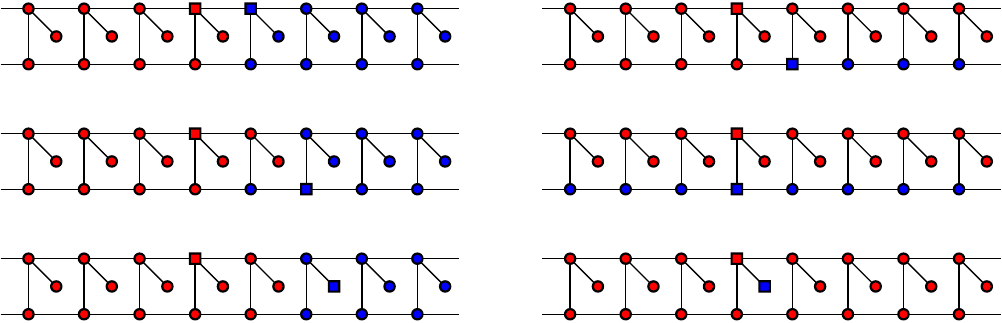}
\caption{Examples of discrete half-spaces in an infinite quasi-transitive Gromov hyperbolic graph. The red vertices represent the half-space $H_G(a,b)$, where $a$ is the red square and $b$ is the blue square, and the blue vertices represent the complement of $H_G(a,b)$. The three left-hand half-spaces are proper whereas the three right-hand half-spaces are not. }
\label{fig:proper}
\end{figure}


Now suppose that $G$ is a bounded degree Gromov hyperbolic graph, and let $\Phi:V\to X$ be a $(\lambda,k)$-rough similarity from $G$ to a closed convex set $X \subseteq \mathbb{H}^d$. 
Then for each $a,b\in V$, the image $\Phi H_G(a,b)$ of the discrete half-space $H_G(a,b)$ is contained in the set
\[
\Bigl\{ x \in \mathbb{H}^d : d(x,\Phi(a)) \leq d(x,\Phi(b))+2k\Bigr\} 
\]
and contains the set
\[
\Bigl\{ x \in \mathbb{H}^d : d(x,\Phi(a)) \leq d(x,\Phi(b))-2k\Bigr\}.
\]
If $k>0$ these sets are not half-spaces in $\mathbb{H}^d$.  However, it follows from the discussion in \cref{subsec:Hdbackground} that if $d(\Phi(a),\Phi(b)) \geq 2k$ and we consider the infinite geodesic in $\mathbb{H}^d$ passing through $\Phi(a)$ and $\Phi(b)$, consider the pair of points on this geodesic at distance $2k$ from $\Phi(b)$, and take $\Phi(b)^-$ to be the point of this pair closer to $\Phi(a)$ and $\Phi(b)^+$ to be the point of the pair further from $\Phi(a)$, then we have that
\begin{equation}
\label{eq:halfspacecomparison}
\Phi ^{-1} \left[H\bigl(\Phi(a),\Phi(b)^-\bigr)\right] \subseteq  H_G(a,b) \subseteq \Phi^{-1} \left[H\bigl(\Phi(a),\Phi(b)^+\bigr)\right].
\end{equation}
Together with \cref{lem:concretevisual} below, which implies the corresponding statement for convex subsets of $\mathbb{H}^d$ that are visible from infinity,  this leads straightforwardly to the following basic fact about half-spaces in $G$.

\begin{lemma}
\label{lem:distantimpliesproper}
Let $G$ be a bounded degree, Gromov hyperbolic graph that is visible from infinity. Then there exists a constant $C$ such that if $a,b\in V$ have $d(a,b)\geq C$ then the discrete half-space $H_G(a,b)$ is proper.
\end{lemma}

A further basic fact about half-spaces that will be important to us is given by the following lemma, which is a simple corollary of \cref{prop:density}.

\begin{lemma}
\label{lem:movinghalfspaces}
Let $G$ be a Gromov hyperbolic graph and let $\Gamma$ be a quasi-transitive subgroup of $\Aut(G)$ that does not have the fixed set property. Then the following hold:
\begin{enumerate}
	\item For every proper discrete half-space $H$ of $G$, there exists an automorphism $\gamma \in \Gamma$ such that  $H$ and $\gamma H$ are disjoint.
	\item
For every pair of proper discrete half-spaces $H_1,H_2$ of $G$ and every finite set of vertices $K$ of $G$, there exists an automorphism $\gamma \in \Gamma$ such that $\gamma (H_1 \cup K)  \subseteq  H_2$.

\end{enumerate}
\end{lemma}

\begin{proof}
We prove the first item, the second being similar.
Let $H$ be a proper discrete half-space of $G$, so that there exists a non-empty open set $U$ in $V \cup \delta G$ that is disjoint from the closure of $H$ in $V \cup \delta G$.
By \cref{prop:density}, there exists a hyperbolic element $\gamma \in \Gamma$ with forward and backward fixed points $\xi$ and $\eta$ both in $U$. Since $\gamma^n x \to \xi$ as $n\to\infty$ uniformly on compact subsets of $V \cup \delta G \setminus \{\eta\}$, it follows that $\gamma^n H \subseteq U$ for sufficiently large $n$, and hence that $\gamma^n H$ and  $H$ are disjoint for all sufficiently large $n$.
\end{proof}

\subsection{Comparing continuum and discrete half-spaces}

In this section, we prove the following comparison between discrete and continuum half-spaces. 

\begin{lemma}
\label{lem:halspacesmaincomparison}
Let $G=(V,E)$ be a bounded degree, Gromov hyperbolic graph that is visible from infinity, and let $\Phi:V\to X$ be a $(\lambda,k)$-rough similarity from $G$ to some closed convex set $X \subseteq \mathbb{H}^d$ that is equal to the convex hull of its boundary. Then there exists a constant $C$ such that for every half-space $H$ in $\mathbb{H}^d$ with $H \cap X \notin\{\emptyset,X\}$ and every $v\in V$ with $d(\Phi(v),H \cap X)\geq C$, there exists $u\in V$ such that the discrete half-space $H_G(u,v)$ is proper, contains $\Phi^{-1} H$, and has
\[\lambda^{-1} d\bigl(\Phi(v),H \cap X\bigr) - C \leq d\bigl(v,H_G(u,v)\bigr) \leq \lambda^{-1} d\bigl(\Phi(v),H \cap X\bigr) + C.\]
\end{lemma}

We begin with the following simple geometric lemma.

\begin{lemma}
\label{lem:concretevisual}
Let $d\geq 2$ and let $X$ be a closed convex subset of $\mathbb{H}^d$ that is equal to the convex hull of its boundary. Then for every $x,y\in X$, there exists an infinite geodesic $\gamma$ in $X$ starting at $x$ such that $d(y,\gamma) \leq \log(1+\sqrt{2})$.
\end{lemma}

It is not hard to see by considering the case that $X$ is an ideal triangle in $\mathbb{H}^2$ that the constant $\log(1+\sqrt{2})$ cannot be improved.

\begin{figure}
\centering
 \includegraphics[height=0.22\textwidth]{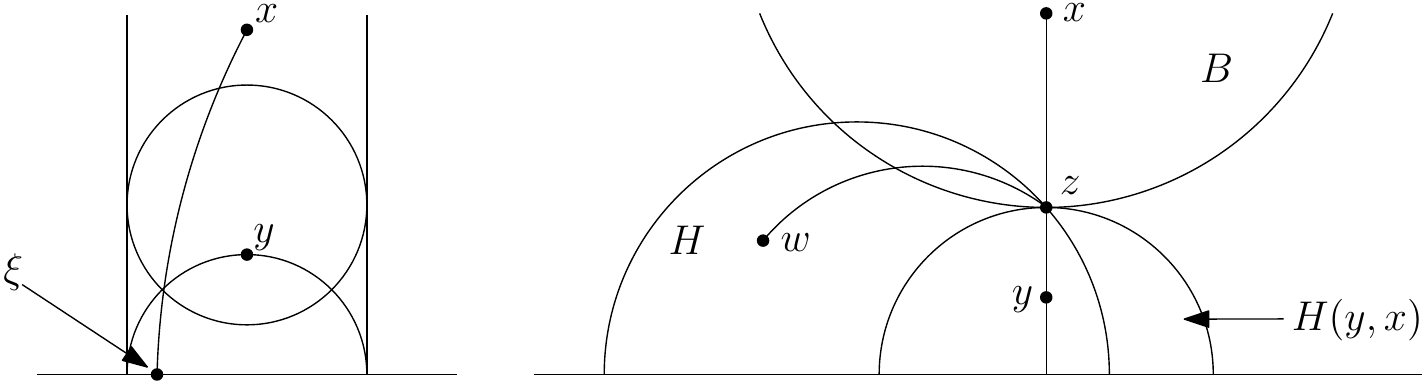}
\caption{Illustrations of the proof of \cref{lem:concretevisual} (left) and of the first part of the proof of \cref{lem:hspacesHtoX} (right).}
\label{fig:comparisonproof}
\end{figure}

\begin{proof}
We work in the Poincar\'e half-space model. 
By applying an isometry of $\mathbb{H}^d$ if necessary, we may assume that $x=(0,\ldots,0,x_d)$ and $y=(0,\ldots,0,1)$ for some $x_d>1$. Since $X$ is the convex hull of its boundary, there exist $\xi,\zeta \in \R^{d-1} \cup \{\infty\}$ such that the hyperbolic geodesic between $\xi$ and $\zeta$ passes through $y$. At least one of these points, say $\xi$, must lie in the closed unit disc in $\R^{d-1}$, and the geodesic from $x$ to $\xi$ is necessarily contained in the closed cylinder lying above this unit disc. The Euclidean ball of radius $1$ centred at $(0,\ldots,0,\sqrt{2})$ is tangent to this cylinder, and coincides with the hyperbolic ball of radius $\log(1+\sqrt{2})$ about $y$. The geodesic from $x$ to $\xi$ must pass through this ball, and the claim follows. See \cref{fig:comparisonproof} for an illustration. 
\end{proof}

\begin{lemma}
\label{lem:hspacesHtoX}
There exists a universal constant $C$ such that the following holds.
Let $d\geq 2$ and let $X$ be a closed convex subset of $\mathbb{H}^d$ that is equal to the convex hull of its boundary.  For every half-space $H$ in $\mathbb{H}^d$ with $H \cap X \notin\{\emptyset,X\}$ and every $x\in X \setminus H$ with $d(x,H\cap X) \geq C$ there exists $y\in X$ such that $H \cap X \subseteq H(y,x)$ and $d(x,H(y,x)) = d(x,y)/2 \geq d(x,H \cap X)-C$.
\end{lemma}

We will prove the claim with the constant  $C=\log(17+12\sqrt{2})$, which is not optimal.

\begin{proof}
Let $z$ be the point of $\partial H \cap X$ closest to $x$, and let $y \in \mathbb{H}^d$ be the unique point with $d(x,y)=2d(x,z)=2d(z,y)$. Note however that $y$ need not be in $X$.  
It is clear that $d(x,H(y,x))=d(x,H \cap X)$. We claim that $H\cap X \subseteq H(y,x)$. 
Indeed, by applying an isometry of $\mathbb{H}^d$ if necessary, it suffices to consider the case that $x=(0,\ldots,0,x_d), y=(0,\ldots,0,y_d),$ and $z=(0,\ldots,0,z_d)$ with $x_d>z_d>y_d$, so that $\partial H(y,x)$ is represented by the Euclidean sphere that is orthogonal to $\R^{d-1}$ and has highest point $z$. In this case, the hyperbolic ball of radius $d(x,z)$ around $x$ is equal to a Euclidean ball $B$ whose boundary sphere passes through $z$ and has its center on the vertical axis, and hence is tangent to the sphere representing $\partial H(y,x)$. 
 If $w\in H \setminus H(y,x)$, then the infinite geodesic passing through $z$ and $w$ has its highest point strictly higher than $z$. Thus, the tangent to the circle representing this geodesic has a positive vertical component at $z$, so that a point a small way along this geodesic from $w$ to $z$ is contained in the interior of the ball $B$. Since $X \cap H$ is convex and $z$ was defined to be the closest point to $x$ in $H$, we deduce that $H \cap X \setminus H(y,x)=\emptyset$ as claimed. See \cref{fig:comparisonproof} for an illustration.

Unfortunately, we do not necessarily have that $y\in X$, and consequently are not yet done. 
 Suppose that $d(x,X \cap H)\geq \log(17+12\sqrt{2})$ and hence that $d(x,H(y,x)) =d(x,z) \geq \log(17+12\sqrt{2})$. 
By \cref{lem:concretevisual}, there exists $z'\in X$ such that $z'$ lies on an infinite geodesic in $X$ starting at $x$, and $d(z,z') \leq \log(1+\sqrt{2})$, so that $d(x,z') \geq \log(17+12\sqrt{2})-\log(1+\sqrt{2})=\log(7+5\sqrt{2})$. Let $y'$ be the unique point in $\mathbb{H}^d$ that has $d(x,y')=2d(x,z')-2\log(2+\sqrt{2})$ and $d(z',y')=d(x,z')-2\log(2+\sqrt{2})$. The point $y'$ lies on the infinite geodesic from $x$ passing through $z'$, and so is in $X$ by choice of $z'$.

\begin{figure}
\centering
\includegraphics[height=0.385\textwidth]{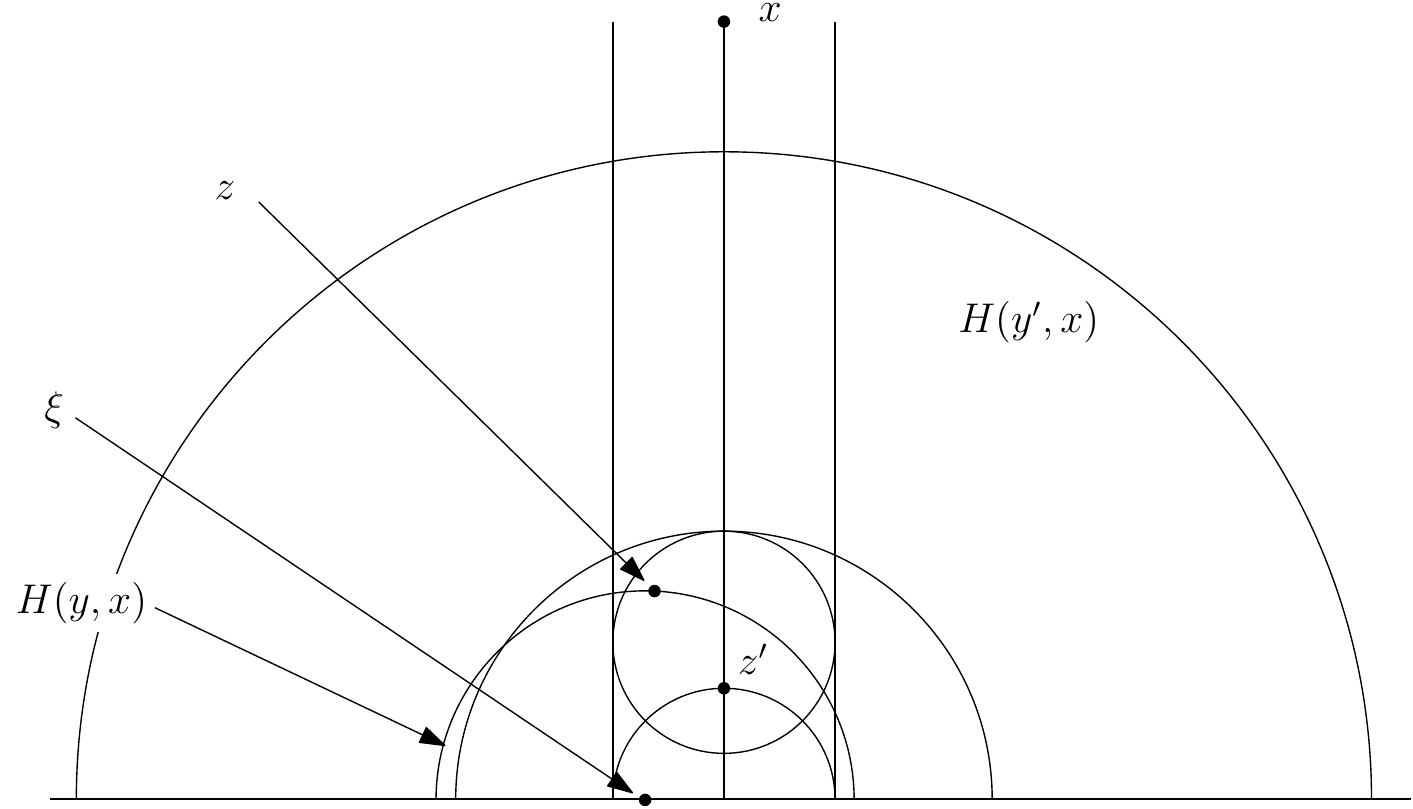}
\caption{Illustration of the second part of the proof of \cref{lem:hspacesHtoX}.}
\vspace{-0.1em}
\end{figure}

Thus, to complete the proof, it suffices to establish that the half-space $H(y',x)$ contains the half-space $H(y,x)$ and that $d(x,H(y',x))\geq d(x,H(y,x))-\log(7+5\sqrt{2})$.
 To see this, apply an isometry of $\mathbb{H}^d$ so that $x,y',z'$ all lie on the vertical axis and $z'=(0,\ldots,0,1)$, so that $z$ lies in the Euclidean ball of radius $1$ centred at $(0,\ldots,0,\sqrt{2})$ and $x_d \geq 2$. Let $\xi \in \R^{d-1}$ be the endpoint of the  geodesic in $\mathbb{H}^d$ starting at $x$ and passing through $z$. Since $x_d>1$, $\xi$ is in the ball of radius $1+\sqrt{2}$ about the origin in $\R^{d-1}$. 
The half-space $H(y,x)$ is represented by the Euclidean ball centred at $\xi$ and whose boundary contains $z$. This ball has radius at most $(1+\sqrt{2})\sqrt{2}=2+\sqrt{2}$, and we deduce that $H(y,x)$ is contained in the Euclidean ball of radius $3+2\sqrt{2}$ about the origin in $\R^d$, which represents the half-space $H(y',x)$ by choice of $y'$. The distance $d(x,H(y',x))$ is equal to $d(x,z')-\log(3+2\sqrt{2})$, which is at least $d(x,z)-\log(3+\sqrt{2})-\log(1+\sqrt{2})=d(x,X \cap H)-\log(7+5\sqrt{2})$.  \qedhere

\end{proof}

\begin{proof}[Proof of \cref{lem:halspacesmaincomparison}]
Let $C'$ be the constant from \cref{lem:hspacesHtoX}. Suppose that $H$ is a half-space in $\mathbb{H}^d$ with $H \cap X \neq \emptyset$ and that $v\in V$ is such that $d(\Phi(v),X \cap H) \geq C'$. Then it follows by \cref{lem:hspacesHtoX} that there exists $y\in X$ such that $d(\Phi(v),H(y,\Phi(v))) = d(\Phi(v),y)/2 \geq d(x,H \cap X) - C'$.  Let $C''$ be a large constant to be chosen and let $y'$ lie on the geodesic from $\Phi(v)$ to $y$ and satisfy $d(\Phi(v),y')=d(\Phi(v),y)-C''$.  If $C''$ is sufficiently large, then whenever $d(\Phi(v),X \cap H)\geq 2C''$ and $y''$ is within distance $k$ of $y'$, the set
\[ \{z \in \mathbb{H}^d : d(z,y') \leq d(z,x) -2k \}\]
contains the half-space $H(y,x)$. Let $u$ be a vertex of $G$ such that $d(\Phi(u),y') \leq k$. If $d(\Phi(u),H)\geq C' \vee C''$, then 
\[
 |\lambda d(v,u) - d(\Phi(v),y)| \leq 
|\lambda d(v,u) - d(\Phi(v),\Phi(u))| + |d(\Phi(v),\Phi(u)) - d(\Phi(v),y)| \leq 2k + C'',
\]
and it follows by choice of $C''$ that $H_G(u,v) \supseteq \Phi^{-1} H=\Phi^{-1} H(y,x)$.
\end{proof}

\subsection{Non-degeneracy}

In our analysis of percolation, it will be convenient for us to place an additional geometric constraint of \emph{non-degeneracy} on the space $X$ in the Bonk-Schramm Theorem. In this subsection we introduce this property, show that is may always be assumed in the quasi-transitive setting, and then give an alternative characterisation of the property in \cref{lem:nondegenerate2}.

Let $X$ be a convex subset of $\mathbb{H}^d$. We write $B_X(x,r)$ for the ball of radius $r$ around $x$ in $X$, and write $B_{\mathbb{H}^d}(x,r)$ for the ball of radius $r$ around $x$ in $\mathbb{H}^d$, so that $B_X(x,r)=B_{\mathbb{H}^d}(x,r) \cap X$. We say that $X$ is \textbf{non-degenerate} if for every $r<\infty$ there exists $R<\infty$ such that for every $x\in X$ and every hyperplane $\partial H$ in $\mathbb{H}^d$, we have that
\[B_X(x,R) \nsubseteq \bigcup_{y\in \partial H} B_{\mathbb{H}^d}(y,r).\]
In other words, $X$ is non-degenerate if it does not contain arbitrarily large balls that are uniformly well-approximated by subsets of hyperplanes. If $X$ is not non-degenerate we say that it is degenerate. Simple examples of degenerate $X$ include geodesics between boundary points, and sets of the form $\partial H \cup K$ where $\partial H \subseteq \mathbb{H}^d$ is a hyperplane and $K \subseteq \mathbb{H}^d$ is compact.

\begin{lemma}
\label{lem:nondegenerate}
Let $G$ be an infinite, quasi-transitive, Gromov hyperbolic graph. Then there exists a natural number $d$ such that $G$ is roughly similar to a non-degenerate closed convex subset $X$ of $\mathbb{H}^d$ that is the convex hull of its boundary.
\end{lemma}

Note that if $G$ is nonamenable then we must have $d\geq 2$ in this lemma.

\begin{proof}
Let $d\geq 1$ be minimal such that there exists a rough similarity from $G$ to some closed convex subset of $\mathbb{H}^d$, and let $\Phi:V\to X$ be a $(\lambda,k)$-rough similarity from $G$ to some closed convex subset $X$ of $\mathbb{H}^d$ for some $\lambda\in (0,\infty)$ and $k\in [0,\infty)$. As discussed in \cref{subsec:BonkSchramm}, we may assume that $X$ is equal to the convex hull of its boundary. We claim that $X$ must be non-degenerate. If $d=1$ this holds trivially  since a convex subset of $\mathbb{H}^1$ is degenerate if and only if it is bounded. 

Suppose then that $d\geq 2$. If $X$ is degenerate, then there exists $r<\infty$ such that for every $R<\infty$ there exists $x_R \in X$ and a hyperplane $\partial H_R $ in $\mathbb{H}^d$ such that $B_X(x_R,R) \subseteq \bigcup_{y\in \partial H_R} B_{\mathbb{H}^d}(y,r)$. Let $v_0$ be a fixed vertex of $G$.  
 Since $G$ is quasi-transitive, there exists a constant $C$ such that for each $x \in X$, there exists $\gamma_x \in \Aut(G)$ such that $d(\Phi (\gamma_x (v_0)),x)\leq C$. Let $\gamma'_x$ be an isometry of $\mathbb{H}^d$ mapping $\Phi( \gamma_x(v_0))$ to $(0,\ldots,0,1)$.
The set of functions $\phi: V \to \mathbb{H}^d$ satisfying $\phi(v_0)=(0,\ldots,0,1)$ and 
\begin{equation}
\label{eq:similaritycompactness}
\Bigl|d\bigl(\phi (u), \phi(v)\bigr )-\lambda d(u,v)\Bigr|\leq k \qquad \text{ for all $u,v \in V$}
\end{equation}
is compact, and so we may take a subsequential limit of the rough similarities $\gamma_{x_R}' \circ \Phi \circ \gamma_{x_R}$ (all of which lie in this set) to obtain a function $\Phi' : V \to \mathbb{H}^d$ satisfying \eqref{eq:similaritycompactness} 
and for which there exists a hyperplane $\partial H$ in $\mathbb{H}^d$ such that the entire set $\Phi' V$ is contained in the $(r+C)$-neighbourhood of $\partial H$. Let $\Psi(v)$ be the closest point in $\partial H$ to $\Phi'(v)$ for each $v\in V$. Then $\Psi$ satisfies 
\[
\Bigl|d\bigl(\Psi (u), \Psi(v)\bigr )-\lambda d(u,v)\Bigr|\leq k + 2r+2C
\]
for every $u,v\in V$. If we identify $\partial H$ with $\mathbb{H}^{d-1}$ and let $Y$ be the closed convex hull of $\Psi(v)$ in $\mathbb{H}^{d-1}$, then $\Psi$ is a rough similarity from $G$ to $Y$. This contradicts the minimality of $d$.
\end{proof}

The following characterisation of non-degeneracy will be particularly useful.

\begin{lemma}
\label{lem:nondegenerate2}
Let $d\geq 2$, identify $\mathbb{H}^d$ with $\R^d_+$ via the Poincar\'e half-space model, and let $X \subseteq \mathbb{H}^d$ be a closed, convex, non-degenerate subset of $\mathbb{H}^d$ that is equal to the convex hull of its boundary.
There exists a constant $C$ such that for every hyperplane $\partial H \subseteq \mathbb{H}^d$ and every $x\in \partial H \cap X$ 
there exists $\xi \in \R^{d-1}$ such that $\xi \in \delta X$, $\| x - \xi \|\leq C x_d$ and the Euclidean distance between $\xi$ and $\partial H$ is at least $C^{-1}x_d$. 
\end{lemma}

\begin{proof}
\begin{figure}
\centering
\includegraphics[trim ={1.5cm 0 1.5cm 0}, clip, height=0.255\textwidth]{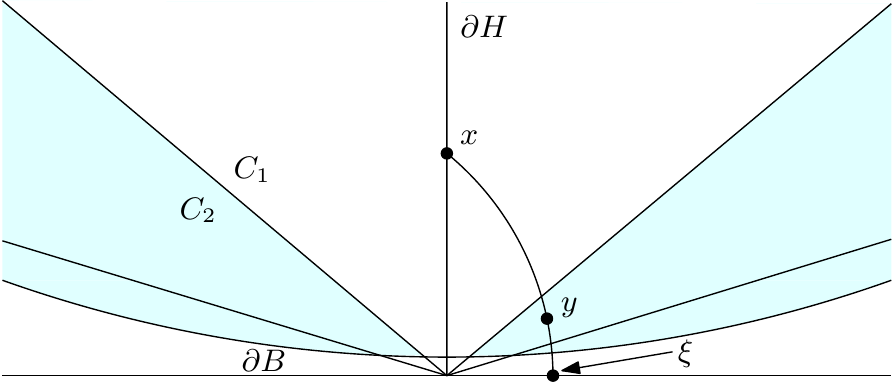}
\hspace{1cm}
\includegraphics[height=0.255\textwidth]{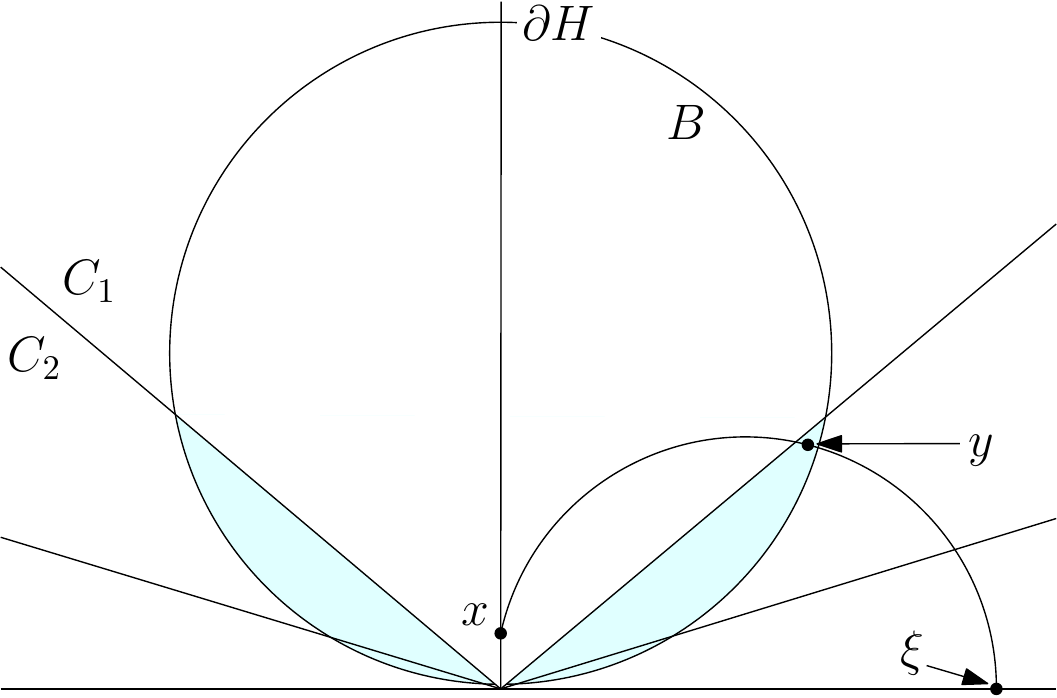}
\caption{Schematic illustration of the proof of \cref{lem:nondegenerate2}. The point $y$ must lie in the blue shaded region, which is equal to $B \setminus C_1$. The limit point of the geodesic from $x$ to $y$ is bounded away from $\partial H$ (left), and is within a bounded distance of $x$ (right).}
\end{figure}

We will prove the lemma in the case that $\infty$ is in the boundary of $\partial H$, so that $\partial H$ is represented by a Euclidean hyperplane orthogonal to $\R^{d-1}$. The proof of the general case is similar but the details are slightly more involved. 
  In this case, the set of points of hyperbolic distance at most $r$ from $\partial H$ is an infinite conical prism (i.e., the product of a cone in $\R^2$ with $\R^{d-2}$) for each $r>0$.
Let $C_1$ be the closed hyperbolic $1$-neighbourhood of $\partial H$ and let $C_2$ be the closed hyperbolic $(1+\log(1+\sqrt{2}))$-neighbourhood of $\partial H$. 
   Since $X$ is non-degenerate, there exists a constant $R$ such that for every $x\in X \cap \partial H$, the hyperbolic ball of radius $R$ around $x$ contains some point $y'$ that is not in $C_2$. Thus, by \cref{lem:concretevisual}, there exists a point $y$ that lies on an infinite geodesic in $X$ starting from $x$, that is in the hyperbolic ball $B$ of radius $R'=R+\log(1+\sqrt{2})$ around $x$, and that is not in $C_1$. Let $\xi$ be the endpoint of this geodesic. 
  The hyperbolic ball $B$ is represented by the Euclidean ball that has its lowest point at the point $(x_1,x_2,\ldots,e^{-R'}x_d )$ and its highest point at the point $(x_1,x_2,\ldots,e^{R'} x_d )$.
The geodesic from $x$ to $y$ is represented by a circle in $\R^d$ that is orthogonal to $\R^{d-1}$ and intersects $\partial H$ at an angle bounded away from $0$ and $\pi$ by a positive $R$-dependent constant. It follows that there exists an $R$-dependent constant $C$ such that $\|x-\xi\| \leq C x_d$ and the Euclidean distance between $\xi$ and $ \partial H$ is at least $C^{-1} x_d$. \qedhere
\end{proof}

\section{A Hyperbolic Magic Lemma}

\label{sec:Magic}
 

The goal of this section is to prove the following proposition, which will be of central importance to our analysis of percolation in \cref{sec:mainproof}. 
Intuitively, the proposition states that for every finite set of vertices $A$ in a Gromov hyperbolic graph, from the perspective of a typical point of $A$, most of $A$ is contained in either one or two distant half-spaces.

\begin{prop}
\label{prop:GromovMagic}
Let $G=(V,E)$ be a Gromov hyperbolic graph with degrees bounded by a constant $M$, and suppose that $\Phi: V \to X$ is a $(\lambda,k)$-rough similarity from $G$ to some closed convex set $X \subseteq \mathbb{H}^d$ for some $d\geq 1$, $\lambda \in (0,\infty)$ and $k \in [0,\infty)$. Then for every $\eps>0$ there exists a constant $N(\eps)=N_{M,\lambda,k,d}(\eps)$ such that for every finite set $A \subseteq V$ there exists a subset $A' \subseteq A$ with the following properties:
\begin{enumerate}
\item $|A'| \geq (1-\eps) |A|$.
\item For every $v\in A'$, there either exists a half-space $H_1 \subseteq \mathbb{H}^d$ or a pair of half-spaces $H_1,H_2 \subseteq \mathbb{H}^d$   
such that $d(\Phi (v),\bigcup H_i) \geq \eps^{-1}$ and $|A \setminus \Phi^{-1} \bigcup H_i|\leq N(\eps)$.
\end{enumerate} 
\end{prop}


\begin{remark}
 It is possible to deduce an intrinsic version of this result in which there is no embedding specified and the half-spaces are discrete. 
\end{remark}

 The following corollary of \cref{prop:GromovMagic} is very closely related to the fact concerning convex hulls of finite sets of vertices in nonamenable Gromov hyperbolic graphs discussed in \cref{subsec:overview}. The two statements can be deduced from each other (in the connected case) by applying a suitable version of the Supporting Hyperplane Theorem.

\begin{figure}
\centering
\includegraphics[width=0.8\textwidth]{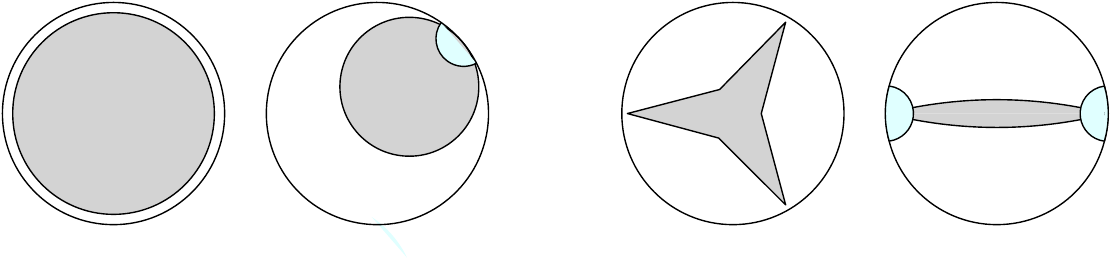}
\caption{Examples of the phenomenon discussed in \cref{sec:Magic}, as depicted in the Poincar\'e disc model. Left: From the perspective of most of its points, a large ball in $\mathbb{H}^d$ looks like a horoball, and most of its volume is contained in a single distant half-space. Right: From the perspective of most of its points, the tripod-like object depicted looks like a thickened line, and most of its volume is contained in two distant half-spaces.}
\end{figure}

\begin{corollary}
\label{cor:emptyhalfspace}
Let $G$ be a bounded degree, nonamenable, Gromov hyperbolic graph. Then there exists $R<\infty$ such that for every finite connected set $A \subseteq V$, there exists a subset $A'\subseteq A$ with $|A'| \geq |A|/2$ such that for every $u \in A'$, there exists $v \in V$ with $d(u,v)\leq R$ such that $H_G(u,v)$ is a proper discrete half-space with $A \subseteq H_G(u,v)$.
\end{corollary} 

We remark that \cref{cor:emptyhalfspace} still holds without the hypothesis that $A$ is connected, but with a longer proof.
The proof of \cref{cor:emptyhalfspace} is given at the end of this section.

We will deduce \cref{prop:GromovMagic} from the following proposition, which establishes a similar statement for $\mathbb{H}^d$. 
We say that a set of points $A \subseteq \mathbb{H}^d$ is $c$-\textbf{separated} if $d(x,y)\geq c$ for every two distinct $x,y\in A$. 

\begin{prop}
\label{prop:HdMagic}
Let $d\geq 1$ and let $c>0$. 
 Then for every $\eps>0$ there exists a constant $N(\eps)=N_{d,c}(\eps)$ such that for every finite $c$-separated set $A \subseteq \mathbb{H}^d$ there exists a subset $A' \subseteq A$ with the following properties:
\begin{enumerate}
\item $|A'| \geq (1-\eps) |A|$.
\item For every $x\in A'$, there either exists a half-space $H_1$ or a pair of half-spaces $H_1,H_2$ such that $d(x,\bigcup H_i) \geq \eps^{-1}$ and $|A \setminus \bigcup H_i|\leq N(\eps)$.
\end{enumerate} 
\end{prop}

We will deduce \cref{prop:HdMagic} from the so-called \emph{Magic Lemma} of Benjamini and Schramm \cite[Lemma 2.3]{BeSc}, which is a related statement for sets of points in Euclidean space. 
Benjamini and Schramm stated their lemma for $\R^2$, but the proof applies to $\R^d$ for every $d\geq 1$. (In fact, Gill \cite{MR3266996} proved that a version of the lemma holds for any doubling metric space.) 

Let $A$ be a finite set of points in $\R^d$ for some $d\geq 1$. For each $x\in A$, the \textbf{isolation radius} $\rho_x$ of $x$ is defined to be $\rho_x=\inf\{\|x-y\| : y \in A \setminus \{x\}\}$. Given $x\in A$, $\delta\in (0,1)$ and $s\geq 2$, we say that $x$ is \textbf{$(\delta,s)$-supported}\footnote{To understand how the Magic Lemma is typically used, the reader may find it helpful to think of $(\delta,s)$-supported points as \textbf{bad} and points that are \emph{not} $(\delta,s)$-supported as \textbf{good}.} if 
\[
\inf_{y\in \R^d} \Bigl| A \cap \bigl(B(x,\delta^{-1} \rho_x) \setminus B(y,\delta \rho_x)\bigr)\Bigr| \geq s.\]
In other words, $x$ is \emph{not} $(\delta,s)$-supported if there exists $y\in \R^d$ such that all but at most $s-1$ points of $A$ are contained in either $B(y,\delta \rho_x)$ or $\R^d \setminus B(x,\delta^{-1} \rho_x)$. 

\begin{prop}[Benjamini-Schramm Magic Lemma] 
\label{prop:OriginalMagic}
Let $d\geq 1$. Then for every $\delta\in (0,1)$ there exists a constant $C=C_d(\delta)$ such that for every finite set $A\subseteq \R^d$, at most $C|A|/s$ of the points of $A$ are $(\delta,s)$-supported.
\end{prop}

Intuitively, this lemma states that for any finite set $A \subseteq \R^d$, from the perspective of a typical point of $A$, most of the points of $A$ are either far away from the point or are contained in a single small, nearby, high-density area. 

\begin{proof}[Proof of \cref{prop:HdMagic}]
Identify $\mathbb{H}^d$ with the open half-space in $\R^d$ using the Poincar\'e half-space model. We begin with some simple  preliminary geometric calculations. 
For each $x\in \mathbb{H}^d$ and $r>1$, we define $H_1(x,r x_d)$ to be the hyperbolic half-space whose complementary half-space $H^c_1(x,r x_d)$ is represented by the Euclidean ball that is orthogonal to $R^{d-1}$ and has its highest point at $(x_1,\ldots,x_{d-1},\sqrt{r^2-1} x_d)$. In other words, $H_1(x,r x_d)$ is the smallest half-space in $\mathbb{H}^d$ that contains the complement $\mathbb{H}^d \setminus B(x,rx_d)$ of the Euclidean ball $B(x,r x_d)$. If $r\geq \sqrt{2}$ then the hyperbolic distance from $x$ to $H_1(x,r x_d)$ is $\log \sqrt{r^2-1}$.
 Similarly, for each $x,y \in \mathbb{H}^d$ and $\delta>0$, we define 
$H_2(y,\delta x_d)$ to be the hyperbolic half-space
  that is defined to be represented by the Euclidean ball that is orthogonal to $\R^{d-1}$ and has its highest point at $(y_1,\ldots,y_{d-1},y_d+\delta x_d)$, so that $H_2(y, \delta x_d)$ is the smallest half-space in $\mathbb{H}^d$ that contains $B(y,\delta x_d) \cap \mathbb{H}^d$. Observe that if $y_d \leq 2 \delta x_d $ and $\delta\leq 1/3$ then the hyperbolic distance between $x$ and $H_2(y,\delta x_d)$ is at least $\log 1/\sqrt{3\delta}$. See \cref{fig:magicproof} for an illustration of these two half-spaces. 

Note also that for every $c$-separated set $A\subseteq \mathbb{H}^d$, the number of points of $A$ in a set $B$ is bounded above by the ratio of the hyperbolic volume of the hyperbolic $c/2$-neighbourhood of $B$ to the hyperbolic volume of a hyperbolic ball of radius $c/2$. Let $C_1$ be the hyperbolic volume of the hyperbolic $c/2$-neighbourhood of the Euclidean ball $B(x,x_d/2)$, which does not depend on the choice of $x\in \mathbb{H}^d$, and let $C_2$ be the ratio of $C_1$ to the volume of a hyperbolic ball of radius $c/2$. 

\begin{figure}
\centering
\includegraphics[height=0.35\textwidth]{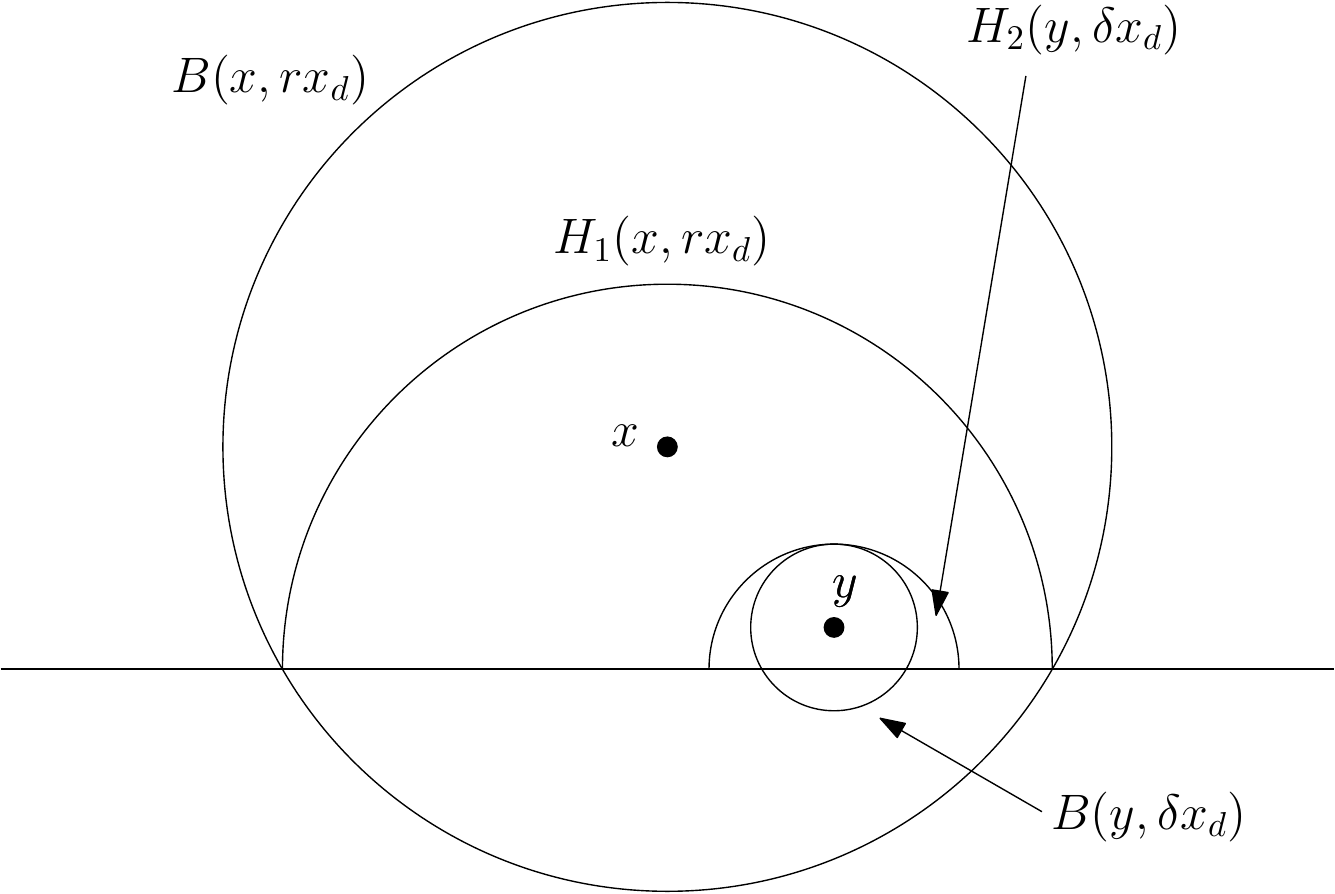}
\caption{Illustration of the Euclidean balls and hyperbolic half-spaces involved in the proof of \cref{prop:HdMagic}.}
\label{fig:magicproof}
\end{figure}

Let $c,\eps>0$ and let $A$ be a finite $c$-separated subset of $\mathbb{H}^d$. Since $A$ is $c$-separated with respect to the hyperbolic metric, the Euclidean isolation radius $\rho_x$ is at least $c' x_d$ for every $x\in A$, where $0<c'=1-e^{-c} \leq 1$. 
Take $\delta=\delta(\eps,c)>0$ so that
\[\min\Bigl\{\log 3^{-1/2}\delta^{-1/4},\log \sqrt{c'^2 \delta^{-2}-1}\Bigr\} \geq \eps^{-1}.
\]
 Let $C_d(\delta)$ be the Magic Lemma constant, let $s=\lceil C_d(\delta)\eps^{-1}\rceil$, and let $N(\eps)=N_{d,c}(\eps)= s+1+C_2$.
Let $A'$ be the set of elements of $A$ that are \emph{not} $(\delta,s)$-supported (with respect to the Euclidean metric), so that, by the Magic Lemma, 
$|A'|\geq(1-\eps)|A|$.
Let $x\in A'$. We claim that the following holds:
\begin{equation}
\label{eq:claim}
\begin{array}{l}\text{There exists either a half-space $H_1$ or a pair of half-spaces $H_1,H_2$}\\ \text{such that $d(x,\bigcup H_i) \geq \eps^{-1}$ and $|A \setminus \bigcup H_i| \leq N(\eps)$. }
\end{array}
\end{equation}
If $\rho_x\geq \delta^{-1/2} x_d$ then we may simply take the single half-space $H_1=H_1(x,\rho_x)$, since for this half-space we have $d(x,H_1)\geq \eps^{-1}$ and $|A\setminus H_1|=1$, and hence that \eqref{eq:claim} holds with this choice of $H_1$. Thus, it suffices to assume that $c'x_d\leq \rho_x \leq \delta^{-1/2}x_d$. 
 Let $y\in \mathbb{R}^d$ be such that $|A \cap (B,x\delta^{-1} \rho_x)\setminus B(y,\delta \rho_x)| \leq s$. Then the half-space $H_1=H_1(x,\delta^{-1}\rho_x)$ has 
$d(x,H_1) \geq \log \sqrt{c'^2\delta^{-2}-1} \geq \eps^{-1}$, and one of the following holds:
\begin{enumerate}
	\item If $y_d\geq 2 \delta \rho_x$, then the hyperbolic $c/2$-neighbourhood of the Euclidean ball $B(y,\delta\rho_x)$ has volume at most $C_1$, and hence $|A\cap B(y,\delta\rho_x)|$ is bounded above by $C_2$.
	\item If $y_d\leq 2 \delta \rho_x \leq 2\delta^{1/2} x_d$, then the half-space $H_2=H_2(y,\delta \rho_x)$ has hyperbolic distance at least $\log 3^{-1/2} \delta^{-1/4}$ from $x$. 
\end{enumerate}
In the first case we take only the first half-space $H_1$, whereas in the second case we take both $H_1$ and $H_2$. In both cases we have that $d(x,\bigcup H_i)\geq \eps^{-1}$ and $|A\setminus \bigcup H_i| \leq N(\eps)$ as claimed.
\end{proof}

\begin{proof}[Proof of \cref{prop:GromovMagic}]
Let $\eps>0$. Let the vertex degrees of $G$ be bounded by $M$.
Since balls of radius $r$ in $G$ contain at most $M^{r+1}$ vertices, it follows that, since $\Phi$ is a rough-similarity, there exists a constant $C$ such that every unit ball in $\mathbb{H}^d$ contains at most $C$ points of $\Phi(V)$. 
Let 
\[\delta = \min\left\{\frac{\eps}{3+\eps}, \frac{\eps}{C}\right\} 
\qquad \text{and let} \qquad
N(\eps) = CN_{d,1}(\delta).\]
Let $A \subset V$ be finite, and let $K\subseteq A$ be maximal such that $\Phi$ is injective on $K$ and $\Phi(K)$ is $1$-separated. That is, $K$ is such that $\Phi$ is injective on $K$, $\Phi(K)$ is $1$-separated, and every point of $\Phi(A)$ is contained in the open $1$-neighbourhood of $\Phi(K)$. For each $v\in A$, let $v_K(v) \in K$ be a vertex in $K$ such that $d(v_K(v),v) < 1$. Applying \cref{prop:HdMagic} to $\Phi(K)$, we deduce that there exists a subset $K'$ of $K$ such that $|K'| \geq (1-\delta) |K|$ and for every $v \in K'$, there either exists a half-space $H_{1,v}$ or a pair of half-spaces $H_{1,v},H_{2,v}$ in $\mathbb{H}^d$ such that $d(\Phi(v),\bigcup H_i) \geq \delta^{-1}$ and $|\Phi(K) \setminus \bigcup H_i| \leq N_{d,1}(\delta)$. It follows from the discussion in \cref{subsec:Hdbackground} that for each $i$, there exists a half-space $H_{i,v}'$ in $\mathbb{H}^d$ such that $d(\Phi(v),H_{i,v}')=d(\Phi(v),H_{i,v})-2$ and $H_{i,v}'$ contains the closed $1$-neighbourhood of $H_{i,v}$ in $\mathbb{H}^d$, so that
\[
\Bigl|A \setminus \bigcup \Phi^{-1} H_{i,v}'\Bigr| \leq \left|\left\{v\in A : v_k(v) \in K \setminus \bigcup \Phi^{-1} H_{i,v} \right\}\right| \leq C N_{d,1}(\delta) = N(\eps).
\]
Letting $A'=\{v\in A : v_K(v) \in K'\}$, we have that \[|A\setminus A'| \leq C |K \setminus K'| \leq C \delta |K| \leq C \delta |A| \leq \eps |A|\]
and that for every $v\in A'$, there exists either a half-space $H_1=H_{1,v_K(v)}'$ or a pair of half-spaces $H_1=H_{1,v_K(v)}', H_2=H_{2,v_K(v)}'$ such that $|A \setminus \bigcup \Phi^{-1} H_{i}| \leq C N_{d,1}(\delta)$ and $d(\Phi(v),H_i) \geq d(\Phi(v_K(v)),H_i) -2 \geq \delta^{-1}-3 \geq \eps^{-1}$. This completes the proof. \qedhere

\end{proof}


\begin{proof}[Proof of \cref{cor:emptyhalfspace}]

Let $G=(V,E)$ be a bounded degree, nonamenable, Gromov hyperbolic graph, and, applying the Bonk-Schramm Theorem, let $\Phi: V \to X$ be a $(\lambda,k)$-rough similarity from $G$ to some closed convex set $X \subseteq \mathbb{H}^d$ that is equal to the convex hull of its boundary for some $d\geq 1, \lambda \in(0,\infty)$, and $k\in [0,\infty)$. 

\begin{figure}
\centering
\includegraphics[height=0.18\textwidth]{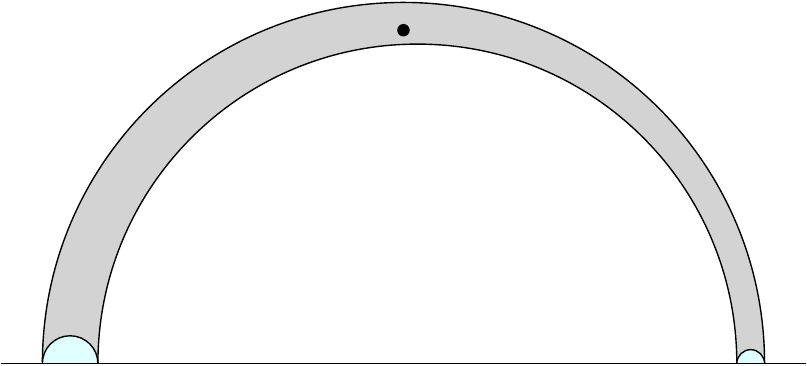}
\hspace{1.275cm}
\includegraphics[height=0.18\textwidth]{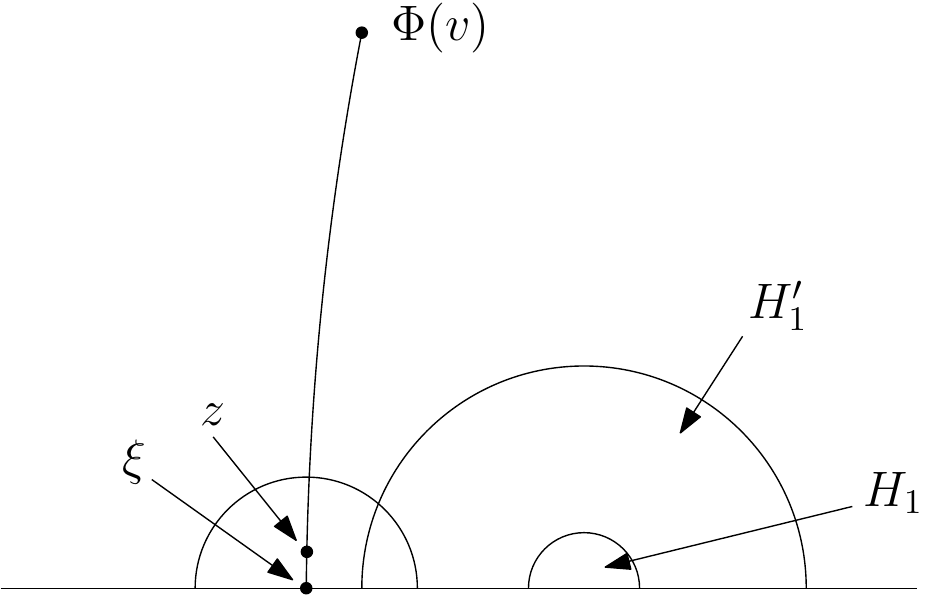}
\caption{Illustration of the proof of \cref{cor:emptyhalfspace}. Left: If the boundary of $X$ is contained in the union of two half-spaces that are both far away from $\Phi(v)$, then $G$ has a large ball around $\Phi(v)$ that is rough-isometric to a line segment. Right: Illustration of the half-spaces that are discussed.}
\label{fig:magicproof2}
\vspace{-0.5em}
\end{figure}

First note that there exists a constant $\eps_0 \leq 1$ such that whenever $v\in V$ and $H_1,H_2$ are any two half-spaces in $\mathbb{H}^d$ with $d(\Phi(v),\bigcup H_i) \geq \eps_0^{-1}$, then there exists $\xi\in \delta X$ that is not in the closure of $\bigcup H_i$ in $\R^{d-1} \cup \{\infty\}$.  Indeed, if this were not the case then $G$ would contain arbitrarily large regions roughly similar to line segments (with uniform constants), contradicting the nonamenability of $G$. See \cref{fig:magicproof2} for an illustration. (This claim also follows, and is in fact equivalent to, the fact that the Gromov boundaries of nonamenable Gromov hyperbolic graphs are uniformly perfect \cite{martinez2017cheeger}.)

Now suppose that $H_1=H(y_1,\Phi(v))$ and $H_2=H(y_2,\Phi(v))$ are half-spaces in $\mathbb{H}^d$ such that $d(\Phi(v),\bigcup H_i) \geq 2\eps_0^{-1}$. Let $y_i'$ be the point that lies on the geodesic from $\Phi(v)$ to $y_i$ and has distance $2\eps^{-1}_0$ from $\Phi(v)$. Let $H_1'=H(y_1',\Phi(v))$ and $H_2'=H(y_2',\Phi(v))$, so that $d(\Phi(v),\bigcup H_i') \geq \eps_0^{-1}$. By the above discussion, there exists $\xi \in \delta X$ that is not in the closure of $\bigcup H'_i$ in $\R^{d-1} \cup \{\infty\}$. Let $r_0$ be a large constant to be chosen and let $z(v,r)$ be the point that lies at distance $r\geq r_0$ from $\Phi(v)$ along the geodesic from $r$ to $\xi$. If $r_0$ is chosen to be  sufficiently large (depending on the choice of $\eps_0$), then the half-space $H(z(v,r),\Phi(v))$ is disjoint from $\bigcup H_i$. (The constant $r_0$ can be chosen independently of $\xi$ thanks to the `buffer regions' $H_1'$ and $H_2'$. Indeed, the worst case occurs when $\xi$ lies in the boundary of $H_1'$ or $H_2'$.)

Now, let $A \subseteq V$ be finite and connected, and apply \cref{prop:GromovMagic} with $\eps=\eps_0/2 \leq 1/2$. Let $A'\subseteq A$ be as in the statement of that proposition. Then for every $v\in A'$ there exists either a half-space $H_1\subseteq \mathbb{H}^d$ or a pair of half-spaces $H_1,H_2\subseteq \mathbb{H}^d$ such that $d(\Phi(v),\bigcup H_i) \geq \eps^{-1}$ and $|A \setminus \Phi^{-1} \bigcup H_i|\leq N(\eps)$. Thus, applying the above discussion with this choice of half-spaces, there exists $r_0$ such that if $r\geq r_0$ then the half-space $H(z(v,r),\Phi(v))$ is disjoint from $\bigcup H_i$. On the other hand, since 
$A$ is connected and $\Phi$ is a rough similarity, there exists a constant $r_1$ such that if $A \cap H(z(v,r+r_1),\Phi(v)) \neq \emptyset$ then $A \cap [H(z(v,r),\Phi(v)) \setminus H(z(v,r+r_1),\Phi(v))] \neq \emptyset$ and consequently that
\vspace{-0.3em}
\[
|A \cap H(z(v,r + r_1),\Phi(v))| \leq \max\left\{ 0,\, |A \cap H(z(v,r),\Phi(v))|-1\right\}
\]
for every $r\geq r_0$. Since $|A \setminus \Phi^{-1}\bigcup H_i|\leq N(\eps)$, it follows that $A$ is disjoint from $H(z(v,r),\Phi(v))$ for every $r\geq  r_0 + N(\eps)r_1$. Since $z(v,r) \in X$, the proof may easily be concluded by an application of \cref{lem:halspacesmaincomparison}. \qedhere


\end{proof}


\begin{remark}
\cref{prop:HdMagic} can also be used to prove that if $\langle G_n \rangle_{n\geq 1}$ is a sequence of finite, $\delta$-hyperbolic graphs with degrees bounded by $M$ for some $\delta\geq 0$ and $M<\infty$ independent of $n$, then any subsequential Benjamini-Schramm limit of $\langle G_n \rangle_{n\geq 1}$ has at most two points in its Gromov boundary almost surely.  This application is similar in spirit to the original use of the Magic Lemma to study circle packings of Benjamini-Schramm limits of finite planar graphs \cite{BeSc}.
\end{remark}

\section{Analysis of percolation}
\label{sec:mainproof}

In this section, we complete the proofs of our main results, namely \cref{thm:pcpu,thm:triangle,thm:pell2}. Using \cref{prop:criterion}, it suffices to prove the following two propositions, which are proven in \cref{subsec:susceptibility} and \cref{subsec:isoperimetriccondition} respectively.

\begin{prop}
\label{thm:susceptibility}
Let $G$ be a unimodular, quasi-transitive, Gromov hyperbolic, nonamenable graph. Then
\[
\limsup_{p\uparrow p_c} (p_c-p)\overline{\chi}_p <\infty.
\]
\end{prop}

\begin{prop}
\label{prop:hypiTp}
Let $G$ be a unimodular, quasi-transitive, Gromov hyperbolic, nonamenable graph. Then
\[
\lim_{p \uparrow p_c} \sup\left\{ \frac{\sum_{u,v\in K}\tau_p(u,v)}{\overline{\chi}_p |K|} : K \subseteq V \text{ finite}\right\} =0.
\]
\end{prop}

Let us now briefly indicate why these two propositions imply our main theorems.

\begin{proof}[Proof of \cref{thm:pcpu,thm:triangle,thm:pell2} given \cref{thm:susceptibility,prop:hypiTp}]
If $G$ is nonunimodular then we have that $p_c(G)<p_{q\to q}(G)$ for every $q\in (1,\infty)$ by \cref{thm:nonunimodular}. If $G$ is unimodular, then combining \cref{thm:susceptibility,prop:hypiTp} yields that, since $p_c(G)<1$, 
\[\liminf_{p\uparrow p_c} \frac{p_c-p}{1-p}\overline{\chi}_p \sqrt{1-\iota(T_p)^2}=0.\] Applying \cref{prop:criterion} we deduce that $p_c(G)<p_{2\to2}(G)$.  Applying \cref{lem:diagrams,prop:2to2givesqtoq}, \cref{thm:pcpu,thm:triangle} follow immediately.
\end{proof}

Let us note that the proof  of \cref{prop:hypiTp} uses the full power of \cref{prop:GromovMagic}, whereas the proof
of \cref{thm:susceptibility} uses only \cref{cor:emptyhalfspace}.

\subsection{The susceptibility exponent}

\label{subsec:susceptibility}

In this section we apply \cref{cor:emptyhalfspace} to prove \cref{thm:susceptibility}. The argument is similar to that appearing in \cite[\S 5]{madras2010}. 
The key input is the following lemma, which is an easy consequence of \cref{cor:emptyhalfspace}. We write $\chi_p=\E_p\left[ |K_\rho|\right]$ for the expected value of the cluster of the random root $\rho$ as discussed in \cref{subsec:unimodularity}. 

\begin{lemma}
\label{lem:susceptibilityhalfspace}
Let $G$ be a unimodular, quasi-transitive, Gromov hyperbolic, nonamenable graph. Then there exist positive constants $c$ and $r$ such that for each $0\leq p<p_c$ there exists 
 $v,u\in V$ with $d(v,u)\leq r$ such that $H_G(v,u)$ is proper and
\[
\mathbf{E}_p\Bigl[ |K_v| \mathbbm{1}\bigl(K_v \subseteq H_G(v,u)\bigr) \Bigr] \geq c\chi_p.
\]
\end{lemma}


\begin{proof}[Proof of \cref{lem:susceptibilityhalfspace}]
For each $v\in V$, $A\subseteq V$, and $r\geq 1$, define
\vspace{-0.5em}
\begin{multline*}I(v,A,r)=\\\mathbbm{1}\Bigl( \text{there exists $u$ with $d(u,v)\leq r$ such that $H_G(v,u)$ is proper and $A \subseteq H_G(v,u)$}\Bigr).
\end{multline*}
\cref{cor:emptyhalfspace} is equivalent to the statement that there exists a constant $r$ such that 
\[
\sum_{v\in A} I(v,A,r) \geq \frac{1}{2}|A|
\]
for every finite $A\subseteq V$.
It follows from the mass-transport principle that
\[
\E_p\left[ |K_\rho| I(\rho,K_\rho,r\bigr) \right] = \E_p\left[ \sum_{v\in K_\rho} I(v,K_\rho,r) \right] \geq \frac{1}{2} \chi_p
\]
for every $p<p_c$. It follows trivially that there exists a vertex $v\in V$ such that
\[
\sum_{d(u,v)\leq r} \mathbf{E}_p\Bigl[ |K_v| \mathbbm{1}\bigl(H_G(v,u) \text{ proper, } K_v \subseteq H_G(v,u)\bigr) \Bigr] \geq 
\mathbf{E}_p\left[ |K_v| I(v,K_v,r\bigr) \right] \geq \frac{1}{2}\chi_p,
\]
and hence that there exists $u$ with $d(u,v)\leq r$ such that $H_G(v,u)$ is proper and
\[
\mathbf{E}_p\left[ |K_v| \mathbbm{1}\bigl(K_v \subseteq H_G(v,u)\bigr)\right] \geq \frac{1}{2\sup_{v\in V}|B_G(v,r)|} \chi_p. \qedhere
\]
\end{proof}

Before beginning the proof of \cref{thm:susceptibility}, let us recall the notion of \emph{Dini derivatives} and some elementary facts about their calculus, and then state Russo's formula. Let $f$ be a measurable function defined on $[0,1]$. The \textbf{lower right and upper right Dini derivatives} of $f$ are defined respectively to be
\[
\lrDini f(p) : = \liminf_{\eps\downarrow 0} \frac{f(p+\eps)-f(p)}{\eps} \quad \text{ and } \quad \urDini f(p):=\limsup_{\eps\downarrow 0} \frac{f(p+\eps)-f(p)}{\eps}.
\]
If $f$ is decreasing then
\[
f(p_2)-f(p_1) \leq \int_{p_1}^{p_2} \lrDini f(p) \dif p \leq \int_{p_1}^{p_2} \urDini f(p) \dif p
\]
for every $p_1<p_2$. Finally, we have that, similarly to the usual chain rule,
\[
\urDini \frac{1}{f(p)} = -\frac{1}{f(p)^2}\lrDini f(p).
\]
See e.g.\ \cite{MR1390758} for further background on Dini derivatives.

We now recall Russo's formula. Recall that, given an increasing event $A\subseteq \{0,1\}^E$ and a percolation configuration $G[p]$, an edge $e$ of $G$ is said to be \textbf{pivotal} for the event $A$  if $A$ occurs in the configuration obtained from $G[p]$ by changing the status of $e$ to be open, but $A$ does not occur in the configuration
obtained from $G[p]$ by changing the status of $e$ to be closed.
Suppose that $A$ is an increasing event depending on only finitely many edges. 
\textbf{Russo's formula} \cite[Theorem 2.25]{grimmett2010percolation} states that $\bP_p(A)$ is differentiable and that
\[
\frac{d}{dp} \bP_p(A)=\sum_{e\in E}\bP_p(e \text{ is pivotal for $A$}) = \frac{1}{1-p}\sum_{e\in E}\bP_p(e \text{ is closed and pivotal for $A$}).
\] 
On the other hand, if $A$ is an increasing event depending on \emph{infinitely} many edges, then Russo's formula ceases to be an equality in general, and we obtain instead the lower right Dini derivative lower bound
\[
\lrDini \bP_p(A) \geq \sum_{e\in E}\bP_p(e \text{ is pivotal for $A$})
=\frac{1}{1-p}\sum_{e\in E}\bP_p(e \text{ is closed and pivotal for $A$}).
\]
See e.g.\ \cite[eq.\ 2.28]{grimmett2010percolation}.

\begin{proof}[Proof of \cref{thm:susceptibility}]
It suffices to show that there exists a positive constant $C$ such that
\begin{equation}
\label{eq:susceptibility_desired}
\lrDini \chi_p \geq C^{-1} \chi_p^2
\end{equation}
for all $p_c/2 \leq p <p_c$. The conclusion of the theorem will then follow 
since we can integrate this differential inequality to obtain that
\[
\chi_p^{-1} = -\left[\chi_{p_c}^{-1}-\chi_{p}^{-1}\right] \geq -\int_p^{p_c} \urDini \chi_p^{-1} dp = \int_p^{p_c} \chi_p^{-2} \lrDini \chi_p dp \geq C^{-1}(p_c-p)
\]
for every $p_c/2<p<p_c$. We conclude since 
\[
\overline{\chi}_p \leq  \sup_{v\in V} \left[\P([\rho]=[v])^{-1}\right]\chi_p
\] 
for every $p\in [0,1]$. 

We now begin to work towards a proof of a differential inequality of the form \eqref{eq:susceptibility_desired}. 
It follows from Russo's formula that
\[
\lrDini\chi_p(v) \geq \frac{1}{1-p}\sum_{e\in E^\rightarrow} \sum_{u\in V} \bP_p\Bigl(\{v \leftrightarrow e^-\}\cup\{ e^- \nleftrightarrow e^+\}\cup\{ e^+ \leftrightarrow u\}\Bigr)
\]
for every $0<p<p_c$ and $v\in V$. (In fact it is not difficult to show that $\chi_p(v)$ is differentiable on $(0,p_c)$ and that this inequality is an equality, but we will not need this.) Averaging over $v$ and applying the mass-transport principle, we deduce that
\begin{equation}
\label{eq:derivative}
\lrDini \chi_p \geq \frac{1}{1-p}\sum_{e^-=\rho} \sum_{u,v\in V} \bP_p\Bigl(\{v \leftrightarrow \rho\}\cup\{ \rho \nleftrightarrow e^+\}\cup\{ e^+ \leftrightarrow u\}\Bigr).
\end{equation}

Let the constants $r$ and $c$ be as in \cref{lem:susceptibilityhalfspace}. 
 Since $G$ is quasi-transitive, there are only finitely many isomorphism classes of pairs $(v,u)$ with $d(u,v) \leq r$. Thus, by \cref{lem:movinghalfspaces}, there exists a constant $R$ such that for every $v,u$ such that $H_G(v,u)$ is proper and $d(u,v)\leq r$, there exists an automorphism $\gamma_{v,u}\in \Aut(G)$ such that $\gamma_{u,v} H_G(v,u) \cap H_G(v,u) = \emptyset$ and $d(v,\gamma_{v,u} v) \leq R$.

Now let $0<p<p_c$ and let $v,u$ be as in \cref{lem:susceptibilityhalfspace}. Write $H=H_G(v,u)$ and $\gamma=\gamma_{v,u}$. 
Noting that the events $\{|K_v|=n, K_v\subseteq H\}$ and $\{|K_{\gamma v}|=m, K_{\gamma v} \subseteq \gamma H\}$ depend on disjoint sets of edges and are therefore independent, we have that
\begin{equation}
\label{eq:independence}
\bE_p \left[ |K_v| \cdot |K_{\gamma v}| \cdot \mathbbm{1}\bigl( K_v \subseteq H, K_{\gamma_v u} \subseteq \gamma H \bigr)\right] \geq c \chi_p^2.
\end{equation}
Let $\eta$ be the geodesic from $v$ to $\gamma v$, so that $\eta$ has length at most $R$.
Let $e$ be the first oriented edge of $G$ that is 
 crossed by $\eta$ and is in $\partial_E H$.
We claim that there exists a $p$-dependent constant $c_p$, bounded on compact subsets of $(0,1)$, such that
\begin{equation}
\label{eq:finite_energy}
\bE_p \left[ |K_{e^-}| \cdot |K_{e^+}| \cdot \mathbbm{1}( e^- \nleftrightarrow e^+)\right] \geq c_p
\bE_p \left[ |K_v| \cdot |K_{\gamma v}| \cdot \mathbbm{1}\bigl( K_v \subseteq H, K_{\gamma_v u} \subseteq \gamma H \bigr)\right]
\end{equation}


\noindent
This follows by a standard finite energy argument: Given any percolation configuration in which $K_v\subseteq H$ and $K_{\gamma v} \subseteq \gamma H$, we can modify the configuration, changing the status only of edges that are either contained in or incident to $\eta$, in such a way that the following hold:
\begin{enumerate}
	\item The clusters of $v$ and $\gamma v$ are both at least as large after the modification as they were originally.
	\item In the modified configuration, $v$ is connected to $e^-$ and $\gamma v$ is connected to $e^+$, but $e^-$ is not connected to $e^+$.
\end{enumerate}
Indeed, to perform such a modification, decompose $\eta=\eta_1 \circ e \circ \eta_2$ into the pieces before and after crossing $e$, and make the following changes:
\begin{enumerate}
	\item Open every edge of $\eta_1$ and every edge of $\eta_2$, and close $e$.
	\item Close every edge incident to (but not contained in) $\eta_1$ that does not already have both endpoints connected to $v$ in the original configuration. 
	\item Close every edge incident to (but not contained in) $\eta_2$ that does not already have both endpoints connected to $\gamma v$ in the original configuration.
\end{enumerate}
It is easily verified that this modification has the required properties, and we deduce the claimed inequality \eqref{eq:finite_energy}. 
Putting together \eqref{eq:derivative}, \eqref{eq:independence}, and \eqref{eq:finite_energy}, and using the fact that $p_c(G)<1$, we deduce that there exists a constant $C$ such that
\[
\lrDini \chi_p \geq c c_p \P([\rho]=[e^-])  \chi_p^2 \geq C^{-1} \chi_p^2
\]
for all $p_c/2 \leq p \leq p_c$ as desired.   \qedhere

\end{proof}

\subsection{The expected number of connections to a far hyperplane at criticality}
\label{subsec:isoperimetriccondition}

In this section we prove \cref{prop:hypiTp}. The
rough idea of the proof is that, by \cref{prop:GromovMagic}, for any finite set $A$, most of $A$ is contained in one or two far-away half-spaces, and these half-spaces do not contribute much to the susceptibility. In order to carry out this idea, we will prove the following estimate. The exact form of the bound in \cref{prop:hyperplanedecay} is not important; we only need that the contribution to the susceptibility from a  distant half-space is uniformly small.

\begin{lemma}
\label{prop:hyperplanedecay}
Let $G=(V,E)$ be a unimodular, quasi-transitive Gromov hyperbolic graph, and suppose that $\Phi: V \to X$ is a $(\lambda,k)$-rough similarity from $G$ to some closed convex set $X \subseteq \mathbb{H}^d$ for some $d\geq 2$, $\lambda \in (0,\infty)$ and $k \in [0,\infty)$. Suppose further that $X$ is non-degenerate and is equal to the convex hull of its boundary. Then there exists a constant $C$ such that
\[
\mathbf{E}_{p} | K_v \cap \Phi^{-1} H | \leq \frac{C \chi_p}{d(\Phi(v), H)} 
\]
for every vertex $v$, every half-space $H\subseteq \mathbb{H}^d$, and every $0\leq p <p_c$.
\end{lemma}

Before proving \cref{prop:hyperplanedecay}, let us see how it implies \cref{prop:hypiTp}.

\begin{proof}[Proof of \cref{prop:hypiTp} given \cref{prop:hyperplanedecay}]
By \cref{lem:nondegenerate} there exists $d\geq 2$, $\lambda \in (0,\infty)$, $k\in [0,\infty)$ and a $(\lambda,k)$-rough similarity $\Phi:V\to X$ from $G$ to some closed, convex set $X \subseteq \mathbb{H}^d$ that is non-degenerate and equal to the convex hull of its boundary.
For each $\eps>0$, let $N(\eps)$ be as in \cref{prop:GromovMagic}. For each $0 < p < p_c$, let $\eps_p > 0$ be infimal such that $N(\eps) \leq (\overline{\chi}_p)^{1/2}$, so that $\eps_p\downarrow 0$ as $p\uparrow p_c$. Thus, for each $0< p <p_c$ and each finite set $A \subseteq V$ there exists a set $A' \subseteq A$ such that $|A'|\geq (1-2\eps_p)|A|$ and for each $u \in A'$ there exists either a half-space $H_{1,u}$ or a pair of half-spaces $H_{1,u},H_{2,u}$ in $\mathbb{H}^d$ such that $d(\Phi(u),H_{i,u}) \geq (2\eps_p)^{-1}$ and $|A \setminus \Phi^{-1} \bigcup H_{i,u} | \leq N(2\eps_p) \leq (\overline{\chi}_p)^{1/2}$. Thus, considering the various contributions to $\sum_{u,v \in A} \tau_p(u,v)$ we have that
\begin{align*}
\sum_{u,v\in A} \tau_p(u,v) &\leq \sum_{u\in A \setminus A', v \in V} \tau_p(u,v) + \sum_{u\in A', v\in A \setminus \Phi^{-1}  \bigcup H_{i,u}}\tau_p(u,v) + \sum_{u\in A', v \in \Phi^{-1}  \bigcup H_{i,u}} \tau_p(u,v)\\ 
&\leq |A \setminus A'| \overline{\chi}_p + |A'| N(2\eps_p) + \sum_{v\in A'} \bE_p \bigl| K_v \cap \Phi^{-1} \bigcup H_{i,u} \bigr|,
\end{align*}
and using \cref{prop:hyperplanedecay} to bound the third term we deduce that
\[
\sum_{u,v\in A} \tau_p(u,v) \leq 2\eps_p |A| \overline{\chi}_p + |A| \overline{\chi}_p^{1/2} + 4C\eps_p |A|\overline{\chi}_p,
\]
where $C$ is the constant from \cref{prop:hyperplanedecay}. It follows that
\[
\sup\left\{\frac{\sum_{u,v \in A}\tau_p(u,v)}{\overline{\chi}_p |A|} : A \subset V \text{ finite}\right\} \leq (\overline{\chi}_p)^{-1/2} + (2+4C)\eps_p,
\]
which tends to zero as $p\uparrow p_c$.
\end{proof}

We now begin to work towards the proof of \cref{prop:hyperplanedecay}, beginning with some general considerations.
Let $G$ be a quasi-transitive graph, and let $v_0$ be a fixed root vertex of $G$.  We say that a (not necessarily finite) set $K \subseteq V$ is \textbf{$r$-roughly branching} if it is empty or if there exists a subset $K' \subseteq K$ and a family of automorphisms $\{\gamma_v : v \in K'\} \subseteq \Aut(G)$ with the following properties:
\begin{enumerate}
	\item $K$ is contained in the $r$-neighbourhood of $K'$ in $G$.
	\item
	$d(\gamma_v v_0, v) \leq r$ for every $v\in K'$.
	\item For every $k\geq 1$, if $v_1,\ldots,v_k$ and $u_1,\ldots,u_k$ are two distinct sequences of elements of $K'$,  then $\gamma_{v_k} \circ \gamma_{v_{k-1}} \circ \cdots \circ \gamma_{v_1} v_0 \neq \gamma_{u_k} \circ \gamma_{u_{k-1}} \circ \cdots \circ \gamma_{u_1} v_0$.
\end{enumerate}

For example, suppose that $T$ is a $k$-regular tree. Fix an orientation of $T$ so that every vertex has one parent and $k-1$ children and, for each vertex, fix an ordering of the children. For each two vertices $u,v$ of $T$, there is a unique automorphism $\gamma_{u,v}$ of $T$ that maps $u$ to $v$ and preserves all of the information regarding the orientation and the ordering of the children. If $v_0$ is a fixed root vertex of $T$ and $K$ is the set of descendants of $v_0$ that are $n$ generations below $v_0$, then $K$ is $0$-roughly branching, as can be seen by taking the automorphisms $\gamma_v=\gamma_{v_0,v}$. 

\begin{lemma}
\label{lem:branchingbound}
Let $G$ be a quasi-transitive graph of maximum degree $M$ and let $v_0$ be a fixed root vertex of $G$. If $K \subseteq V$ is $r$-roughly branching for some $r\geq 0$  then
\[
\sum_{v\in K} \tau_p(v_0,v) \leq \left(\frac{M}{p^2} \right)^r
\]
for every $0< p \leq p_c$.
\end{lemma}

\begin{proof}
 Let $K'$ and $\{\gamma_v : v\in K'\}$ be as in the definition of $K$ being $r$-roughly branching. 
 Let $k\geq 1$ and let $(K')^k$ be the set of sequences of elements of $K'$ of length $k$. 
 For each $\mathbf{u}=(u_1,\ldots,u_k) \in (K')^k$, let $\gamma(\mathbf{u}) = \gamma_{u_k} \circ \gamma_{u_{k-1}}\circ \cdots \circ \gamma_{u_1}$. For each $1\leq i \leq k$, let $\gamma_i(\mathbf{u}) = \gamma_{u_i} \circ \cdots \circ \gamma_{u_1}$, and let $\gamma_0(\mathbf{u})$ be the identity automorphism. Then, by the Harris-FKG inequality,
\[
\tau_p(v_0,\gamma(\mathbf{u}) v_0) \geq \prod_{i=1}^{k} \tau_p(\gamma_{i-1}(\mathbf{u}) v_0, \gamma_{i}(\mathbf{u}) v_0) =
\prod_{i=1}^{k} \tau_p(\gamma_{i-1}(\mathbf{u}) v_0, \gamma_{u_i} \gamma_{i-1}(\mathbf{u}) v_0) = 
\prod_{i=1}^{k} \tau_p(v_0, \gamma_{u_i} v_0).
\]
Since $\gamma(\mathbf{u}) v_0 \neq \gamma(\mathbf{u}') v_0$ for distinct $\mathbf{u},\mathbf{u}'\in (K')^k$, it follows that
\[
\chi_p(v_0) \geq \sum_{\mathbf{u}\in (K')^k} \tau_p(v_0,\gamma v_0) \geq  \sum_{\mathbf{u}\in (K')^k} \prod_{i=1}^k\tau_p(v_0,\gamma_{u_i} v_0) =  \left[ \sum_{u\in K'} \tau_p(v_0,\gamma_u v_0)\right]^k.
\]
When $p<p_c$, the left-hand side is finite and does not depend on $k$. Thus, we deduce that the right-hand side must be bounded in $k$, which implies that
\[
\sum_{u\in K'} \tau_p(v_0,\gamma_u v_0) \leq 1
\]
for every $0<p<p_c$. Since $\tau_p(u,v)$ is non-negative and left-continuous in $p$ for every $u,v\in V$ (see e.g.\ \cite[Lemma 5]{Hutchcroft2016944}) the same inequality also holds at $p=p_c$.

For each $v\in K$, let $u(v)\in K'$ be such that $d(v,u(v))\leq r$. Then $d(v,\gamma_{u(v)}v_0)\leq 2r$ and hence, by the Harris-FKG inequality,
\[
\tau_p(v_0,v)\leq \tau_p(v,\gamma_{u(v)}v_0))^{-1}\tau_p(v_0,\gamma_{u(v)}v_0)\leq p^{-2r}\tau_p(v_0,\gamma_{u(v)}v_0).
\]
It follows that there exists a constant $C$ such that
\begin{equation*}
\sum_{u\in K} \tau_p(v_0,u) \leq p^{-2r} \sum_{u\in K'} \tau_p(v_0,\gamma_u v_0)|\{v\in K : u(v)=u\}| \leq M^r p^{-2r}
\end{equation*}
for every $0< p \leq p_c$.
\end{proof}

To apply \cref{lem:branchingbound} in our setting, we will show that certain sets arising as discrete approximations to hyperplanes in quasi-transitive Gromov hyperbolic graphs are roughly branching.

For the remainder of this section, we fix a unimodular, quasi-transitive, nonamenable, Gromov hyperbolic graph $G$ with fixed root vertex $v_0$, and take $\Phi: V \to X$ be a $(\lambda,k)$-rough similarity from $G$ to a closed, convex, non-degenerate set $X \subseteq \mathbb{H}^d$ that is equal to the convex hull of its boundary for some $\lambda \in (0,\infty)$, $k\in [0,\infty)$ and $d\geq 2$. Our next goal is to prove the following.

\begin{lemma}
\label{lem:hyperplanesbranching}
There exist positive constants $r$, $R$, and $C \geq k$ such that whenever $n\geq 1$ and $\Phi (v_0)=x_0,x_1,\ldots,x_n$ all lie on a geodesic in $\mathbb{H}^d$ and have $d(x_{0},x_{i})=4ri$ for every $1\leq i \leq n$, the set
\vspace{-0.2em}
\[
\vspace{-0.4em}
K=\bigcup_{i=1}^{n} \left\{ v \in V : d\left(\Phi(v), \partial H\left( x_i , x_{i+1} \right)\right) \leq C  \right\}
\]
 is $R$-roughly branching.
\end{lemma}


The details regarding the choices of constants in the following proof can be dealt with via similar explicit computations to those performed in the proof of \cref{lem:hspacesHtoX}.

\begin{proof}

Since $G$ is quasi-transitive, there exists $C_1$ such that the distance from $u$ to the $\Aut(G)$-orbit of $v_0$ is at most $C_1$ for every $u\in V$. Since $\Phi$ is a rough similarity, it follows that there exists a constant $C_2$ such that for every $x\in X$, the distance in $\mathbb{H}^d$ between $x$ and the image under $\Phi$ of the orbit of $v_0$  is at most $C_2$. We will take the constant $C$ in \cref{lem:hyperplanesbranching} to be this $C_2$.

\begin{figure}
\centering
\includegraphics[width=0.7\textwidth]{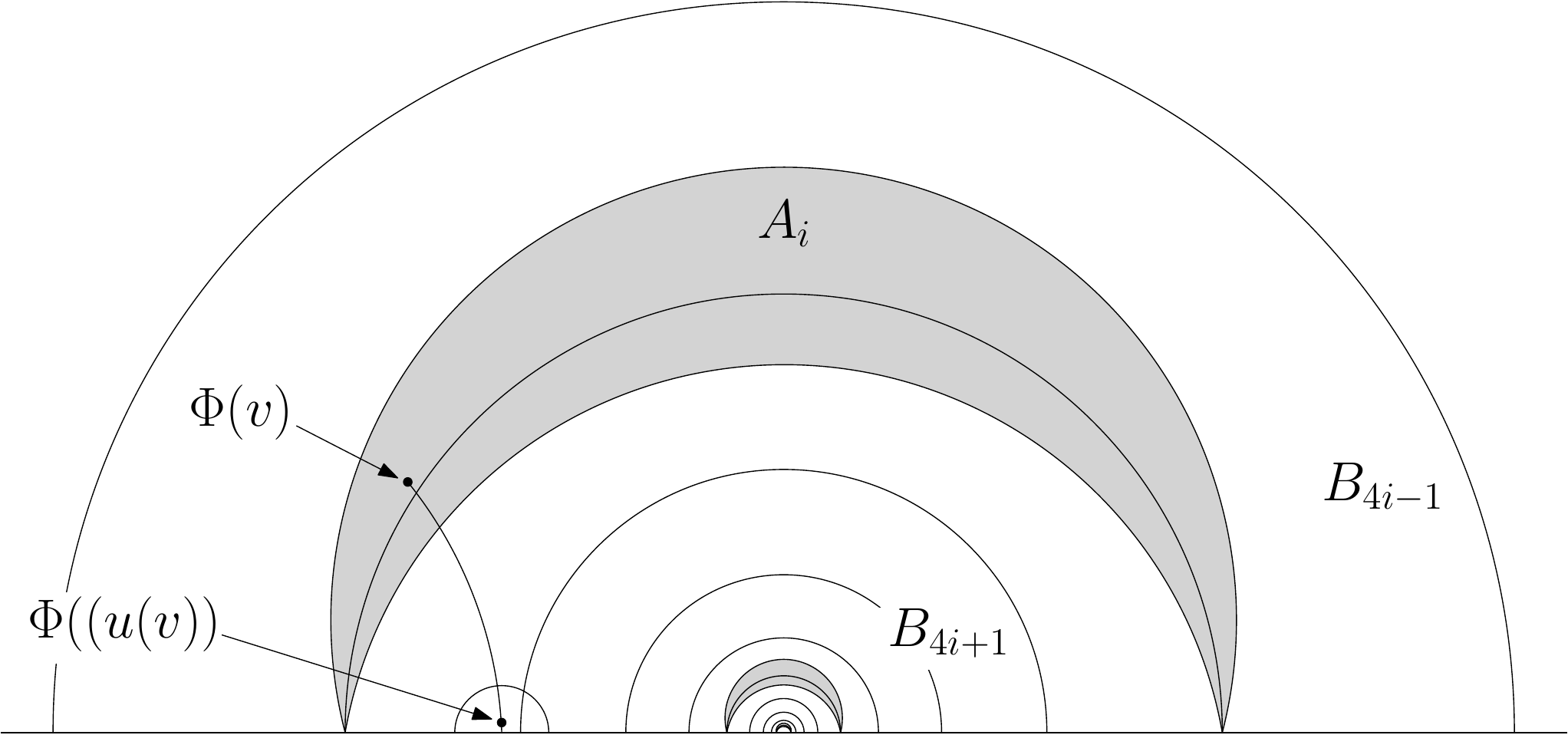}
\caption{Schematic illustration of the some of the sets arising in the proof of \cref{lem:hyperplanesbranching}.}
\label{fig:branching}
\end{figure}

 Let $A_i$ be the hyperbolic $C_2$-neighbourhood in $X$ of $\partial H(x_i,x_{i+1}) \cap X$ for each $1 \leq i \leq n$, and let $K_i=\Phi^{-1} A_i$, so that $K=\bigcup_{i=1}^n K_i$. If $K_i$ is empty then $K_{i+1}$ is also by choice of $C_2$ and by convexity, so we may assume without loss of generality that each of the sets $K_i$ is non-empty, reducing to a smaller value of $n$ or to a triviality otherwise. 
Using the Poincar\'e half-space model and applying an isometry of $\mathbb{H}^d$ if necessary, we may assume that $x_0=\Phi (v_0)=(0,\ldots,0,1)$ and $x_i=(0,\ldots,0,e^{-4ri})$ for each $i=1,\ldots,n$, so that the hyperplane $\partial H(x_i,x_{i+1})$ is represented by the Euclidean sphere orthogonal to $\R^{d-1}$ that has its highest point at $(0,\ldots,0,e^{-4ri-2})$. Let $B_i$ be the half-space in $\mathbb{H}^d$ that has its highest point at $(0,\ldots,0,e^{-ir-2r})$. There exists a constant $r_0$ such that if $r\geq r_0$ then 
$A_i \subseteq B_{4i-1} \setminus B_{4i+1}$ for each $1\leq i \leq n$, so that in particular the sets $A_i$ are disjoint. Since $\Phi^{-1} B_1$ contains $\Phi^{-1} A_1$ is therefore non-empty, we can apply \cref{lem:halspacesmaincomparison} to $B_1$ to deduce that there exist constants $r_1$ and $C_3$ such that if $r\geq r_1$ then there exists a proper discrete half-space $H_G(u_0,v_0)$ such that $H_G(u_0,v_0)$ contains $\Phi^{-1} B_1$ and $d(v_0,u_0)\leq C_3 r$. 


 Since $X$ is non-degenerate and equal to the convex hull of its boundary, we can apply \cref{lem:nondegenerate2} to deduce that there exists a constant $C_4$, that does \emph{not} depend on $r$, such that for each $x \in A_i$ there exists $\xi(x) \in \R^{d-1} \cap \delta X$ such that the Euclidean distance between $x$ and $\xi(x)$ is at most $C_4 x_d$ and the Euclidean distance between $\xi(x)$ and $A_i$ is at least $C_4^{-1}x_d$. It follows by a straightforward trigonometric calculation that there  exist constants $r_2$ and $C_{5}$ such that if $r \geq r_2$ and $x\in A_i$ satisfies 
$x_d\leq C_5^{-1} e^{-4ri-2r}$ 
 then the geodesic from $x$ to $\xi(x)$ is contained in $B_{4i-1} \setminus B_{4i+1}$ for every $1\leq i \leq n$ and $x\in A_i'$. We now fix $r=r_0 \vee r_1 \vee r_2$, and let $A_i'$ be the set of $x \in A_i$ that are in the orbit of $v_0$ and satisfy $x_d \leq C^{-1}_5 e^{-4ri-2r}$.


Let $C_{6}$ be a large constant to be chosen and let $y(x) \in X$ be the point that lies on the geodesic from $x$ to $\xi(x)$ and has $d_\mathbb{H}(x,y(x)) = C_{6}$. If $C_{6}$ is sufficiently large, then the half-space $H(y(x),x)$ is contained in $B_{4i-2} \setminus B_{4i+2}$ and has Euclidean distance at least $C_4^{-1} x_d/2$ from $A_i$ for each $1\leq i \leq n$ and $x\in A'_i$. Taking a suitably large constant $r_3$, we can find for each $1\leq i \leq n$ and $v\in \Phi^{-1} A_i'$ a vertex $u(v)$ such that $H_G(u(v),v)$ is proper, $d(v,u(v))\leq r_3$, and $H_G(u(v),v) \subseteq \Phi^{-1} H(y(\Phi(v)),\Phi(v)) \subseteq \Phi^{-1}(B_{4i+2}\setminus B_{4i-2})$. In particular, if $v_1\in A_i'$ and $v_2\in A_j'$ for $i\neq j$ then the discrete half-spaces $H_G(u(v_1),v_1)$ and $H_G(u(v_2),v_2)$ are disjoint. See \cref{fig:branching} for an illustration. 


If the constant $C_{7}$ is chosen to be sufficiently large then for every pair of distinct points  $x_1,x_2$ in $A_i'$ with $d(x_1,x_2)\geq C_{7}$, the  hyperbolic half-spaces $H(y(x_1),x_1)$ and $H(y(x_2),x_2)$ are disjoint. We take $K_i'$ to be a maximal subset of $K_i$ with the property that $\Phi v \in A_i'$ for every $v\in K_i'$ and such that $\Phi K_i'$ is $C_{7}$-separated, and define $K'=\bigcup_{i=1}^n K_i'$. 
 Thus, we have that there exists a constant $C_{8}$ such that
$K$ is contained in the $C_{8}$ neighbourhood of $K'$ in $G$, and for any two $v_1,v_2 \in K'$, the half-spaces $H_G(u(v_1),v_1)$ and $H_G(u(v_2),v_2)$ are disjoint.


Recall that $u_0$ was defined so that $d(v_0,u_0)\leq C_3 r$ and that $H_G(u_0,v_0)$ is a proper discrete half-space containing  $\Phi^{-1} B_1$. 
For each $v \in V$ that is in the same orbit of $\Aut(G)$ as $v_0$, let $\gamma_{1,v} \in \Aut(G)$ be such that $\gamma_{1,v} v_0=v$. Applying \cref{lem:movinghalfspaces}, we have that for every $v\in K'$ there exists an automorphism $\gamma_{2,v} \in \Aut(G)$ such that $\gamma_{2,v} \gamma_{1,v} (H_G(u_0,v_0) \cup \{v_0\})$ is contained in $H_G(u(v),v)$. Set $\gamma_v = \gamma_{2,v}\gamma_{1,v}$. 
Since there are only finitely many possibilities for the isomorphism class of the tuple $(v,u(v),\gamma_{1,v} v_0, \gamma_{1,v} u_0)$, we can choose the automorphisms $\gamma_{2,v}$ in such a way that there exists a constant $r_4$ such that $d(\gamma_v v_0,v) = d(\gamma_{2,v} v,v)\leq r_4$ for each $v\in  K'$. 

For each $k\geq 1$ and each sequence $\mathbf{v}=(v_1,v_2,\ldots,v_k) \in (K')^k$, let $\gamma(\mathbf{v})=\gamma_{v_k}\circ \cdots \circ \gamma_{v_1}$. 
To complete the proof, it suffices to prove that $\gamma(\mathbf{v}) v_0 \neq \gamma(\mathbf{v}') v_0$ for every $k\geq 1$ and every two distinct $\mathbf{v},\mathbf{v}'\in (K')^n$. 
To do this, we will prove the following claim by induction on $k\geq 1$:
\begin{multline}
\label{eq:bad_typesetting}
 \text{The sets $\gamma(\mathbf{v})\left[H_G(u_0,v_0) \cup \{v_0\} \right]$ and $\gamma(\mathbf{v}')\left[H_G(u_0,v_0) \cup \{v_0\} \right]$}\\\text{are disjoint whenever $\mathbf{v},\mathbf{v}'$ are distinct elements of $(K')^k$.}
\end{multline}
First observe that, by definition of $\gamma_v$, we have that $\gamma_v \left[H_G(u_0,v_0) \cup \{v_0\} \right] \subseteq H_G(u(v),v)$ for every $v\in K'$, and since $H_G(u(v),v) \subseteq H_G(u_0,v_0)$, it follows by induction that
\begin{equation}
\label{eq:induction_middle_step}
\gamma(\mathbf{v}) \left[H_G(u_0,v_0) \cup \{v_0\} \right] \subseteq H_G(u(v_k),v_k)) \subseteq H_G(u_0,v_0) 
\end{equation}
for every $k\geq 1$ and $\mathbf{v}\in (K')^k$. 
Since $H_G(u(v),v)$ and $H_G(u(v'),v')$ are disjoint whenever $v,v'\in K'$ are distinct, this immediately establishes \eqref{eq:bad_typesetting} in the base case $k=1$. Now suppose that \eqref{eq:bad_typesetting} has been established for sequences of length $k$, and suppose that $\mathbf{v},\mathbf{v'} \in (K')^{k+1}$ are distinct.
 If $v_{k+1} \neq v'_{k+1}$, then it follows from \eqref{eq:induction_middle_step} that $\gamma(\mathbf{v}) \left[H_G(u_0,v_0) \cup \{v_0\} \right]  \subseteq H_G(u(v_{k+1}),v_{k+1})$ and $\gamma(\mathbf{v}') \left[H_G(u_0,v_0) \cup \{v_0\} \right]  \subseteq H_G(u(v'_{k+1}),v'_{k+1})$ are disjoint as desired. 
 Otherwise, $v_{k+1}=v'_{k+1}$ and the sequences $\hat{\mathbf{v}}=(v_1,\ldots,v_{k})$ and $\hat{\mathbf{v}} ' = (v'_1,\ldots,v'_{k})$ are distinct. In this case, it follows by the induction hypothesis that 
$\gamma( \hat{\mathbf{v}}) \left[H_G(u_0,v_0) \cup \{v_0\} \right]$ and $\gamma(\hat{\mathbf{v}}') \left[H_G(u_0,v_0) \cup \{v_0\} \right]$ are disjoint, and hence that the sets
$
\gamma( \mathbf{v}) \left[H_G(u_0,v_0) \cup \{v_0\} \right]=\gamma_{v_{k+1}}
\gamma( \hat{\mathbf{v}}) \left[H_G(u_0,v_0) \cup \{v_0\} \right]$ and $\gamma( \mathbf{v}') \left[H_G(u_0,v_0) \cup \{v_0\} \right]=\gamma_{v_{k+1}}
\gamma( \hat{\mathbf{v}}') \left[H_G(u_0,v_0) \cup \{v_0\} \right]$ are disjoint also. This completes the induction.
\end{proof}

\begin{proof}[Proof of \cref{prop:hyperplanedecay}]
Let $C$ and $r$ be as in \cref{lem:hyperplanesbranching}.  Let $v_0$ be a fixed root vertex of $G$. Since $G$ is quasi-transitive, it suffices to prove the claim with constants that may a priori depend on the choice of $v_0$. Let $H$ be a half-space in $\mathbb{H}^d$, and let $n=\lfloor d(\Phi(v_0),H)/4r\rfloor$. Let $\Phi(v_0)=x_0,x_1,\ldots,x_n \in \mathbb{H}^d$ lie on the geodesic from $\Phi(v_0)$ to $H$ and have $d(x_0,x_{i})=4ri$ for each $1\leq i \leq n$. Let
$K_i = \{ u\in V : d(u,\partial H(x_i,x_{i+1})) \leq C \}$.
 By \cref{lem:hyperplanesbranching} and \cref{lem:branchingbound}, there exists a constant $C'$ such that
\[
\sum_{i=1}^n \mathbf{E}_{p_c} |K_{v_0} \cap K_i|  \leq C',
\]
and hence that there exists some $1\leq i \leq n$ such that
$\mathbf{E}_{p_c} |K_{v_0} \cap K_i|  \leq C'n^{-1}$.
Since any path from $v_0$ to $\Phi^{-1} H$ must pass through the set $K_i$ for every $1\leq i \leq n$, it follows by the BK inequality that
\[
\mathbf{E}_{p} |K_{v_0} \cap \Phi^{-1} H |  \leq \mathbf{E}_{p_c} |K_{v_0} \cap K_i|  \overline{\chi}_p \leq C \overline{\chi}_p d(\Phi(v_0),H)^{-1}
\]
as desired.
\end{proof}

\label{subsec:criticaldecay}

\section{Closing discussion and open problems}

\label{sec:closing}

It is natural to conjecture that the following strengthening of \cref{conj:pcpu,conj:triangle} also holds. 

\begin{conjecture}
\label{conj:pell2}
Let $G$ be a connected, locally finite, quasi-transitive, nonamenable graph. Then $p_c(G)<p_{q\to q}(G)$ for every $q\in (1,\infty)$.
\end{conjecture}

By \cref{prop:2to2givesqtoq}, this conjecture is equivalent to the statement that $p_c<p_{2\to 2}$ under the same hypotheses.
We remark that the perturbative proofs of $p_c<p_u$ in \cite{bperc96,MR3005730,MR1756965,MR1833805} also implictly establish the stronger result $p_c<p_{2\to2}$, so that in particular every nonamenable finitely generated group has a Cayley graph with $p_c<p_{2\to2}$. By \cref{thm:nonunimodular}, to prove \cref{conj:pell2} it would suffice to consider the unimodular case. The proofs of \cite{MR1614583,BS00,MR2221157,MR1757952,MR3009109} establish $p_c<p_u$ without establishing $p_c<p_{2\to2}$. Further consequences of \cref{conj:pell2} are investigated in \cite{Hutp2to2}.

The following conjecture would further strengthen \cref{conj:pell2}. The difficult part of the conjecture should be to prove that $p_{q\to q}<p_{2\to 2}$ for every $q\in [1,2)$. It can be deduced from the methods of \cite{1712.04911} that \cref{conj:pqtoq} holds for the product of finitely many trees each of which is regular of some degree $\geq 3$. 

\begin{conjecture}
\label{conj:pqtoq}
Let $G$ be a connected, locally finite, quasi-transitive, nonamenable graph. Then $p_{q\to q}(G)$ is a continuous, strictly increasing function of $q$ on $[1,2]$.
\end{conjecture}

\subsection{Non-uniqueness at $p_{2\to2}$.}

We remark that several estimates concerning critical percolation that are proven using supermultiplicativity arguments as in \cite{Hutchcroft2016944,1712.04911} can easily be adapted to establish related estimates at $p_{2\to2}$ and $p_{q\to q}$. 
For example, the methods of \cite{Hutchcroft2016944} can easily be generalized to establish the following.

\begin{prop}
\label{prop:p2to2kappa}
Let $G$ be a quasi-transitive graph, let $\operatorname{gr}(G) :=\limsup_{n\to\infty} |B(v,n)|^{1/n}$, and let $q\in [2,\infty]$. Then
\[\kappa_p(n):=\inf\left\{\tau_p(u,v) : d(u,v) \leq n\right\} \leq \operatorname{gr}(G)^{-(q-1)n/q}\]
for every $0\leq p \leq p_{q\to q}(G)$.
\end{prop}
The corresponding estimate at $p_c$ states that $\kappa_{p_c}(n)\leq \operatorname{gr}(G)^{-n}$ \cite[Theorem 1.2]{Hutchcroft2016944}.

\begin{proof} By H\"older's inequality we have that
\[
\kappa_p(n) |B(v,n)| \leq \langle T_p \mathbbm{1}_{B(v,n)}, \mathbbm{1}_v \rangle
\leq \|T_p \mathbbm{1}_{B(v,n)}\|_q \|\mathbbm{1}_v\|_{\frac{q}{q-1}}
 \leq \|T_p\|_{q\to q} |B(v,n)|^{1/q},
\]
where the second inequality is by Cauchy-Schwarz. By \cite[Lemma 4]{Hutchcroft2016944} the sequence $\kappa_p(n)$ is supermultiplicative for each $p\in [0,1]$, and it follows by Fekete's Lemma that
\[
\sup_{n\geq 1} \kappa_p(n)^{1/n} = \lim_{n\to\infty} \kappa_p(n)^{1/n} \leq \liminf_{n\to\infty} \|T_p\|_{q \to q}^{1/n} |B(v,n)|^{-(q-1)/qn} \leq \operatorname{gr}(G)^{-(q-1)/q}
\]
for every $0\leq p<p_{q \to q}$. It follows from \cite[Lemma 5]{Hutchcroft2016944} that this estimate still holds at $p_{q \to q}$.
\end{proof}

Note that \cref{prop:p2to2kappa} also implies that $p_{q \to q}(G)\leq \operatorname{gr}(G)^{-(q-1)/q}<1$ for every quasi-transitive nonamenable graph $G$ and every $q\in [2,\infty]$. This bound and that of \cref{prop:p2to2kappa} are both attained for all $q\in [2,\infty]$ when $G$ is a $k$-regular tree for some $k\geq 3$.

Finally, we remark that the methods of \cite{1712.04911} can be used to show that the estimate known as \emph{Schramm's Lemma} continues to apply at $p_{2\to2}$.  The fact that this bound holds at $p_c$ was originally proven for unimodular transitive graphs by Schramm, see \cite{kozma2011percolation}.

\begin{prop}
\label{prop:nonuniquenessatpell2}
Let $G$ be a connected, locally finite, transitive, nonamenable graph, let $X$ be simple random walk on $G$, and let $\rho(G)<1$ be the spectral radius of $G$. Then
\[
\E\left[ \tau_p(X_0,X_n) \right] \leq \rho(G)^{n}
\]
for every $0\leq p \leq p_{2\to2}$.
\end{prop}

\begin{proof} 
We omit some details since the proof is very similar to that of \cite[Theorem 3.1]{1712.04911}.
First suppose that $0\leq p < p_{2\to2}$. 
Let $P$ be the Markov operator for simple random walk on $G$.
We have by Cauchy-Schwarz that
\[
\E\left[ \tau_p(X_0,X_n) \right] = \langle T_p \mathbbm{1}_v, P^n \mathbbm{1}_v\rangle \leq \sqrt{\|T^2_p\|_{2\to2} \|P^{2n}\|_{2\to2}}.
\]
On the other hand, the sequence $\E\left[ \tau_p(X_0,X_n) \right]$ is supermultiplicative and hence, by Fekete's Lemma, satisfies
\[
\sup_{n\geq 0} \E\left[ \tau_p(X_0,X_n) \right] = \lim_{n\to\infty}\E\left[ \tau_p(X_0,X_n) \right]^{1/n} \leq \lim_{n\to\infty} \|T^2_p\|_{2\to2}^{1/2n} \|P^{2n}\|^{1/2n}_{2\to2} = \rho(G).
\]
This completes the proof in the case $0\leq p<p_{2\to2}$. The case $p=p_{2\to2}$ follows by the same left-continuity argument used in \cite{1712.04911}.
\end{proof}


\subsection{Some speculation}

\cref{prop:p2to2kappa} implies in particular that there cannot be a unique infinite cluster at $p_{2\to2}$ on any quasi-transitive nonamenable graph, and hence that any quasi-transitive graph $G$ that has a unique infinite cluster at $p_u$ must have $p_{2\to2}<p_u$. 
This motivates the following problem. 

\begin{question}
Let $G$ be a connected, locally finite, quasi-transitive, nonamenable graph. Under what conditions is $p_u(G)=p_{2\to2}(G)$? 
\end{question}

This is related to the question of which nonunimodular transitive graphs have $p_t=p_h$. 
Several classes of graphs are known to have infinitely many infinite clusters at $p_u$, including nonamenable products \cite{MR1770624} and Cayley graphs of Kazhdan groups \cite{LS99}. Do some of these classes of graphs also have $p_u=p_{2\to2}$? 

What about lattices in $\mathbb{H}^d$?
It is not unreasonable to expect that when $d$ is large, the uniqueness transition for percolation on lattices in $\mathbb{H}^d$ has a mean-field character, that is, behaves similarly to the recurrence/transience transition for branching random walk (BRW) on the same lattices. This transition is now quite well understood following the work of Lalley and Sellke \cite{MR1452555}, Karpelevich, Pechersky, and Suhov \cite{MR1641015}, Lalley and Gou\"ezel \cite{MR3087391}, and Gou\"ezel  \cite{MR3194496}. If this intuition is correct, then we should expect that $p_u=p_{2\to 2}$ and that there is exponential decay of the two-point function at $p_u$. 

\subsection*{Acknowledgments} 
We thank Omer Angel and Jonathan Hermon for helpful discussions. We also thank Itai Benjamini, Elisabetta Candellero, Jonathan Hermon, Gady Kozma, Russ Lyons, Asaf Nachmias, Vincent Tassion, and Henry Wilton for useful comments on earlier versions of the manuscript. We thank Asaf Nachmias in particular for his careful reading of the technical parts of the paper.

  \setstretch{1}
  \bibliographystyle{abbrv}
{\footnotesize{
  \bibliography{unimodularthesis.bib}

\begin{thebibliography}{10}

\bibitem{aizenman1987sharpness}
M.~Aizenman and D.~J. Barsky.
\newblock Sharpness of the phase transition in percolation models.
\newblock {\em Communications in Mathematical Physics}, 108(3):489--526, 1987.

\bibitem{MR901151}
M.~Aizenman, H.~Kesten, and C.~M. Newman.
\newblock Uniqueness of the infinite cluster and continuity of connectivity
  functions for short and long range percolation.
\newblock {\em Comm. Math. Phys.}, 111(4):505--531, 1987.

\bibitem{MR762034}
M.~Aizenman and C.~M. Newman.
\newblock Tree graph inequalities and critical behavior in percolation models.
\newblock {\em J. Statist. Phys.}, 36(1-2):107--143, 1984.

\bibitem{anderson2006hyperbolic}
J.~Anderson.
\newblock {\em Hyperbolic geometry}.
\newblock Springer Science \& Business Media, 2006.

\bibitem{1710.03003}
O.~Angel and T.~Hutchcroft.
\newblock Counterexamples for percolation on unimodular random graphs.
\newblock 2017.
\newblock Preprint.

\bibitem{unimodular2}
O.~Angel, T.~Hutchcroft, A.~Nachmias, and G.~Ray.
\newblock Hyperbolic and parabolic unimodular random maps.
\newblock {\em Geom. Funct. Anal.}, 2018.
\newblock To appear.

\bibitem{MR3692901}
S.~Antoniuk, E.~Friedgut, and T.~{\L}uczak.
\newblock A sharp threshold for collapse of the random triangular group.
\newblock {\em Groups Geom. Dyn.}, 11(3):879--890, 2017.

\bibitem{MR3239613}
S.~Antoniuk, T.~{\L}uczak, and J.~\'Swipolhk~atkowski.
\newblock Collapse of random triangular groups: a closer look.
\newblock {\em Bull. Lond. Math. Soc.}, 46(4):761--764, 2014.

\bibitem{antunovic2008sharpness}
T.~Antunovi{\'c} and I.~Veseli{\'c}.
\newblock Sharpness of the phase transition and exponential decay of the
  subcritical cluster size for percolation on quasi-transitive graphs.
\newblock {\em Journal of Statistical Physics}, 130(5):983--1009, 2008.

\bibitem{MR1622785}
E.~Babson and I.~Benjamini.
\newblock Cut sets and normed cohomology with applications to percolation.
\newblock {\em Proc. Amer. Math. Soc.}, 127(2):589--597, 1999.

\bibitem{MR1127713}
D.~J. Barsky and M.~Aizenman.
\newblock Percolation critical exponents under the triangle condition.
\newblock {\em Ann. Probab.}, 19(4):1520--1536, 1991.

\bibitem{benjamini2016self}
I.~Benjamini.
\newblock Self avoiding walk on the seven regular triangulation.
\newblock {\em arXiv preprint arXiv:1612.04169}, 2016.

\bibitem{MR2970060}
I.~Benjamini and R.~Eldan.
\newblock Convex hulls in the hyperbolic space.
\newblock {\em Geom. Dedicata}, 160:365--371, 2012.

\bibitem{BLPS99b}
I.~Benjamini, R.~Lyons, Y.~Peres, and O.~Schramm.
\newblock Critical percolation on any nonamenable group has no infinite
  clusters.
\newblock {\em Ann. Probab.}, 27(3):1347--1356, 1999.

\bibitem{bperc96}
I.~Benjamini and O.~Schramm.
\newblock Percolation beyond {$\bold Z^d$}, many questions and a few answers.
\newblock {\em Electron. Comm. Probab.}, 1:no.\ 8, 71--82, 1996.

\bibitem{BS00}
I.~Benjamini and O.~Schramm.
\newblock Percolation in the hyperbolic plane.
\newblock {\em J. Amer. Math. Soc.}, 14(2):487--507, 2001.

\bibitem{BeSc}
I.~Benjamini and O.~Schramm.
\newblock Recurrence of distributional limits of finite planar graphs.
\newblock {\em Electron. J. Probab.}, 6:no. 23, 13 pp. (electronic), 2001.

\bibitem{MR1771428}
M.~Bonk and O.~Schramm.
\newblock Embeddings of {G}romov hyperbolic spaces.
\newblock {\em Geom. Funct. Anal.}, 10(2):266--306, 2000.

\bibitem{MR2243589}
B.~H. Bowditch.
\newblock {\em A course on geometric group theory}, volume~16 of {\em MSJ
  Memoirs}.
\newblock Mathematical Society of Japan, Tokyo, 2006.

\bibitem{burton1989density}
R.~M. Burton and M.~Keane.
\newblock Density and uniqueness in percolation.
\newblock {\em Communications in mathematical physics}, 121(3):501--505, 1989.

\bibitem{MR3420526}
P.-E. Caprace, Y.~Cornulier, N.~Monod, and R.~Tessera.
\newblock Amenable hyperbolic groups.
\newblock {\em J. Eur. Math. Soc. (JEMS)}, 17(11):2903--2947, 2015.

\bibitem{MR2986821}
J.~Czajkowski.
\newblock Clusters in middle-phase percolation on hyperbolic plane.
\newblock In {\em Noncommutative harmonic analysis with applications to
  probability {III}}, volume~96 of {\em Banach Center Publ.}, pages 99--113.
  Polish Acad. Sci. Inst. Math., Warsaw, 2012.

\bibitem{1303.5624}
J.~Czajkowski.
\newblock Non-uniqueness phase of bernoulli percolation on reflection groups
  for some polyhedra in $\mathbb{H}^3$.
\newblock 2013.
\newblock arXiv:1303.5624.

\bibitem{1804.05948}
J.~Czajkowski.
\newblock One-point boundaries of ends of clusters in percolation in
  $\mathbb{H}^d$.
\newblock 2018.
\newblock arXiv:1804.05948.

\bibitem{de2015characterizing}
M.~De~La~Salle and R.~Tessera.
\newblock Characterizing a vertex-transitive graph by a large ball.
\newblock {\em arXiv preprint arXiv:1508.02247}, 2015.

\bibitem{MR3551187}
M.~Duchin, K.~Jankiewicz, S.~C. Kilmer, S.~Leli\`evre, J.~M. Mackay, and A.~P.
  S\'anchez.
\newblock A sharper threshold for random groups at density one-half.
\newblock {\em Groups Geom. Dyn.}, 10(3):985--1005, 2016.

\bibitem{1707.00520}
H.~Duminil-Copin.
\newblock Lectures on the ising and potts models on the hypercubic lattice.
\newblock 2017.

\bibitem{duminil2015new}
H.~Duminil-Copin and V.~Tassion.
\newblock A new proof of the sharpness of the phase transition for bernoulli
  percolation and the ising model.
\newblock {\em Communications in Mathematical Physics}, pages 1--21, 2015.

\bibitem{MR3658330}
B.~Federici and A.~Georgakopoulos.
\newblock Hyperbolicity vs. amenability for planar graphs.
\newblock {\em Discrete Comput. Geom.}, 58(1):67--79, 2017.

\bibitem{fitzner2015nearest}
R.~Fitzner and R.~van~der Hofstad.
\newblock Mean-field behavior for nearest-neighbor percolation in {$d>10$}.
\newblock {\em Electron. J. Probab.}, 22:Paper No. 43, 65, 2017.

\bibitem{MR2221157}
D.~Gaboriau.
\newblock Invariant percolation and harmonic {D}irichlet functions.
\newblock {\em Geom. Funct. Anal.}, 15(5):1004--1051, 2005.

\bibitem{gandolfi1992uniqueness}
A.~Gandolfi, M.~Keane, and C.~Newman.
\newblock Uniqueness of the infinite component in a random graph with
  applications to percolation and spin glasses.
\newblock {\em Probability Theory and Related Fields}, 92(4):511--527, 1992.

\bibitem{MR2948665}
A.~Georgakopoulos and M.~Hamann.
\newblock On fixing boundary points of transitive hyperbolic graphs.
\newblock {\em Arch. Math. (Basel)}, 99(1):91--99, 2012.

\bibitem{ghys1990groupes}
E.~Ghys and P.~de~la Harpe, editors.
\newblock {\em Sur les groupes hyperboliques d'apr\`es {M}ikhael {G}romov},
  volume~83 of {\em Progress in Mathematics}.
\newblock Birkh\"auser Boston, Inc., Boston, MA, 1990.
\newblock Papers from the Swiss Seminar on Hyperbolic Groups held in Bern,
  1988. English translation by W.E.\ Grosso available at
  \url{http://perso.ens-lyon.fr/ghys/articles/groupeshyperboliques-english.pdf}.

\bibitem{MR3266996}
J.~T. Gill.
\newblock Doubling metric spaces are characterized by a lemma of {B}enjamini
  and {S}chramm.
\newblock {\em Proc. Amer. Math. Soc.}, 142(12):4291--4295, 2014.

\bibitem{MR3194496}
S.~Gou\"ezel.
\newblock Local limit theorem for symmetric random walks in {G}romov-hyperbolic
  groups.
\newblock {\em J. Amer. Math. Soc.}, 27(3):893--928, 2014.

\bibitem{MR3087391}
S.~Gou\"ezel and S.~P. Lalley.
\newblock Random walks on co-compact {F}uchsian groups.
\newblock {\em Ann. Sci. \'Ec. Norm. Sup\'er. (4)}, 46(1):129--173 (2013),
  2013.

\bibitem{grimmett2010percolation}
G.~R. Grimmett.
\newblock Percolation (grundlehren der mathematischen wissenschaften).
\newblock 2010.

\bibitem{Gromov81}
M.~Gromov.
\newblock Hyperbolic manifolds, groups and actions.
\newblock In {\em Riemann surfaces and related topics: {P}roceedings of the
  1978 {S}tony {B}rook {C}onference ({S}tate {U}niv. {N}ew {Y}ork, {S}tony
  {B}rook, {N}.{Y}., 1978)}, volume~97 of {\em Ann. of Math. Stud.}, pages
  183--213. Princeton Univ. Press, Princeton, N.J., 1981.

\bibitem{Gromov87}
M.~Gromov.
\newblock Hyperbolic groups.
\newblock In {\em Essays in group theory}, volume~8 of {\em Math. Sci. Res.
  Inst. Publ.}, pages 75--263. Springer, New York, 1987.

\bibitem{Haggstrom06}
O.~H\"aggstr\"om and J.~Jonasson.
\newblock Uniqueness and non-uniqueness in percolation theory.
\newblock {\em Probab. Surv.}, 3:289--344, 2006.

\bibitem{MR1676835}
O.~H\"aggstr\"om and Y.~Peres.
\newblock Monotonicity of uniqueness for percolation on {C}ayley graphs: all
  infinite clusters are born simultaneously.
\newblock {\em Probab. Theory Related Fields}, 113(2):273--285, 1999.

\bibitem{HPS99}
O.~H{\"a}ggstr{\"o}m, Y.~Peres, and R.~H. Schonmann.
\newblock Percolation on transitive graphs as a coalescent process: relentless
  merging followed by simultaneous uniqueness.
\newblock In {\em Perplexing problems in probability}, volume~44 of {\em Progr.
  Probab.}, pages 69--90. Birkh\"auser Boston, Boston, MA, 1999.

\bibitem{Hamann2017}
M.~Hamann.
\newblock Group actions on metric spaces: fixed points and free subgroups.
\newblock {\em Abhandlungen aus dem Mathematischen Seminar der Universit{\"a}t
  Hamburg}, 87(2):245--263, Oct 2017.

\bibitem{MR1043524}
T.~Hara and G.~Slade.
\newblock Mean-field critical behaviour for percolation in high dimensions.
\newblock {\em Comm. Math. Phys.}, 128(2):333--391, 1990.

\bibitem{heydenreich2015progress}
M.~Heydenreich and R.~van~der Hofstad.
\newblock {\em Progress in high-dimensional percolation and random graphs}.
\newblock CRM Short Courses. Springer, Cham; Centre de Recherches
  Math\'ematiques, Montreal, QC, 2017.

\bibitem{Hutchcroft2016944}
T.~Hutchcroft.
\newblock Critical percolation on any quasi-transitive graph of exponential
  growth has no infinite clusters.
\newblock {\em Comptes Rendus Mathematique}, 354(9):944 -- 947, 2016.

\bibitem{Hutchcroftnonunimodularperc}
T.~Hutchcroft.
\newblock Non-uniqueness and mean-field criticality for percolation on
  nonunimodular transitive graphs.
\newblock 2017.

\bibitem{Hutp2to2}
T.~Hutchcroft.
\newblock The ${L}^2$ boundedness condition in nonamenable percolation.
\newblock 2018.
\newblock In preparation.

\bibitem{1712.04911}
T.~Hutchcroft.
\newblock Statistical physics on a product of trees.
\newblock {\em Ann. Inst. H. Poincaré Probab. Statist.}, 2018.
\newblock To appear.

\bibitem{MR1390758}
R.~Kannan and C.~K. Krueger.
\newblock {\em Advanced analysis on the real line}.
\newblock Universitext. Springer-Verlag, New York, 1996.

\bibitem{MR1921706}
I.~Kapovich and N.~Benakli.
\newblock Boundaries of hyperbolic groups.
\newblock In {\em Combinatorial and geometric group theory ({N}ew {Y}ork,
  2000/{H}oboken, {NJ}, 2001)}, volume 296 of {\em Contemp. Math.}, pages
  39--93. Amer. Math. Soc., Providence, RI, 2002.

\bibitem{MR1641015}
F.~I. Karpelevich, E.~A. Pechersky, and Y.~M. Suhov.
\newblock A phase transition for hyperbolic branching processes.
\newblock {\em Comm. Math. Phys.}, 195(3):627--642, 1998.

\bibitem{kozma2011percolation}
G.~Kozma.
\newblock Percolation on a product of two trees.
\newblock {\em The Annals of Probability}, pages 1864--1895, 2011.

\bibitem{MR2779397}
G.~Kozma.
\newblock The triangle and the open triangle.
\newblock {\em Ann. Inst. Henri Poincar\'e Probab. Stat.}, 47(1):75--79, 2011.

\bibitem{MR2551766}
G.~Kozma and A.~Nachmias.
\newblock The {A}lexander-{O}rbach conjecture holds in high dimensions.
\newblock {\em Invent. Math.}, 178(3):635--654, 2009.

\bibitem{MR2748397}
G.~Kozma and A.~Nachmias.
\newblock Arm exponents in high dimensional percolation.
\newblock {\em J. Amer. Math. Soc.}, 24(2):375--409, 2011.

\bibitem{MR1614583}
S.~P. Lalley.
\newblock Percolation on {F}uchsian groups.
\newblock {\em Ann. Inst. H. Poincar\'e Probab. Statist.}, 34(2):151--177,
  1998.

\bibitem{MR1873136}
S.~P. Lalley.
\newblock Percolation clusters in hyperbolic tessellations.
\newblock {\em Geom. Funct. Anal.}, 11(5):971--1030, 2001.

\bibitem{MR1452555}
S.~P. Lalley and T.~Sellke.
\newblock Hyperbolic branching {B}rownian motion.
\newblock {\em Probab. Theory Related Fields}, 108(2):171--192, 1997.

\bibitem{MR1757952}
R.~Lyons.
\newblock Phase transitions on nonamenable graphs.
\newblock {\em J. Math. Phys.}, 41(3):1099--1126, 2000.
\newblock Probabilistic techniques in equilibrium and nonequilibrium
  statistical physics.

\bibitem{MR3009109}
R.~Lyons.
\newblock Fixed price of groups and percolation.
\newblock {\em Ergodic Theory Dynam. Systems}, 33(1):183--185, 2013.

\bibitem{LP:book}
R.~Lyons and Y.~Peres.
\newblock {\em Probability on Trees and Networks}.
\newblock Cambridge University Press, New York, 2016.
\newblock Available at \url{http://pages.iu.edu/~rdlyons/}.

\bibitem{LPS06}
R.~Lyons, Y.~Peres, and O.~Schramm.
\newblock Minimal spanning forests.
\newblock {\em Ann. Probab.}, 34(5):1665--1692, 2006.

\bibitem{LS99}
R.~Lyons and O.~Schramm.
\newblock Indistinguishability of percolation clusters.
\newblock {\em Ann. Probab.}, 27(4):1809--1836, 1999.

\bibitem{madras2010}
N.~Madras and C.~Wu.
\newblock Trees, animals, and percolation on hyperbolic lattices.
\newblock {\em Electron. J. Probab.}, 15:2019--2040, 2010.

\bibitem{martinez2017cheeger}
{\'A}.~Mart{\'\i}nez-P{\'e}rez and J.~M. Rodr{\'\i}guez.
\newblock Cheeger isoperimetric constant of gromov hyperbolic manifolds and
  graphs.
\newblock {\em Communications in Contemporary Mathematics}, page 1750050.

\bibitem{MR3005730}
A.~Nachmias and Y.~Peres.
\newblock Non-amenable {C}ayley graphs of high girth have {$p_c<p_u$} and
  mean-field exponents.
\newblock {\em Electron. Commun. Probab.}, 17:no. 57, 8, 2012.

\bibitem{MR923855}
B.~G. Nguyen.
\newblock Gap exponents for percolation processes with triangle condition.
\newblock {\em J. Statist. Phys.}, 49(1-2):235--243, 1987.

\bibitem{MR2205306}
Y.~Ollivier.
\newblock {\em A {J}anuary 2005 invitation to random groups}, volume~10 of {\em
  Ensaios Matem\'aticos [Mathematical Surveys]}.
\newblock Sociedade Brasileira de Matem\'atica, Rio de Janeiro, 2005.

\bibitem{MR1756965}
I.~Pak and T.~Smirnova-Nagnibeda.
\newblock On non-uniqueness of percolation on nonamenable {C}ayley graphs.
\newblock {\em C. R. Acad. Sci. Paris S\'er. I Math.}, 330(6):495--500, 2000.

\bibitem{MR1770624}
Y.~Peres.
\newblock Percolation on nonamenable products at the uniqueness threshold.
\newblock {\em Ann. Inst. H. Poincar\'e Probab. Statist.}, 36(3):395--406,
  2000.

\bibitem{MR1676831}
R.~H. Schonmann.
\newblock Stability of infinite clusters in supercritical percolation.
\newblock {\em Probab. Theory Related Fields}, 113(2):287--300, 1999.

\bibitem{MR1833805}
R.~H. Schonmann.
\newblock Multiplicity of phase transitions and mean-field criticality on
  highly non-amenable graphs.
\newblock {\em Comm. Math. Phys.}, 219(2):271--322, 2001.

\bibitem{MR1888869}
R.~H. Schonmann.
\newblock Mean-field criticality for percolation on planar non-amenable graphs.
\newblock {\em Comm. Math. Phys.}, 225(3):453--463, 2002.

\bibitem{MR3352259}
A.~Thom.
\newblock A remark about the spectral radius.
\newblock {\em Int. Math. Res. Not. IMRN}, (10):2856--2864, 2015.

\bibitem{MR3572426}
A.~Thom.
\newblock The expected degree of minimal spanning forests.
\newblock {\em Combinatorica}, 36(5):591--600, 2016.

\bibitem{timar2006percolation}
{\'A}.~Tim{\'a}r.
\newblock Percolation on nonunimodular transitive graphs.
\newblock {\em The Annals of Probability}, pages 2344--2364, 2006.

\bibitem{tykesson2007number}
J.~Tykesson.
\newblock The number of unbounded components in the poisson boolean model of
  continuum percolation in hyperbolic space.
\newblock {\em Electronic Journal of Probability}, 12:1379--1401, 2007.

\bibitem{MR1245204}
W.~Woess.
\newblock Fixed sets and free subgroups of groups acting on metric spaces.
\newblock {\em Math. Z.}, 214(3):425--439, 1993.

\bibitem{Woess}
W.~Woess.
\newblock {\em Random walks on infinite graphs and groups}, volume 138 of {\em
  Cambridge Tracts in Mathematics}.
\newblock Cambridge University Press, Cambridge, 2000.

\bibitem{yamamoto2017upper}
K.~Yamamoto.
\newblock An upper bound for the critical probability on the cartesian product
  graph of a regular tree and a line.
\newblock {\em arXiv preprint arXiv:1705.06873}, 2017.

\bibitem{MR1995802}
A.~\.Zuk.
\newblock Property ({T}) and {K}azhdan constants for discrete groups.
\newblock {\em Geom. Funct. Anal.}, 13(3):643--670, 2003.

\end{thebibliography}
}
\end{document}